\documentclass{irmaart}
\input{ha.sty}
\usepackage[displaymath]{lineno}
		 \nolinenumbers 

\usepackage[hyperindex,pagebackref,final]{hyperref}

\hypersetup{
		colorlinks = true,
		linkcolor = red,
		anchorcolor = black,
		citecolor = blue, 
		filecolor = cyan,
		menucolor = red,
		runcolor = cyan,
		urlcolor = magenta,
		linkbordercolor = {white},
		linktocpage = true
}

\usepackage{calc}
\newcounter{hours}\newcounter{minutes}

\usepackage[latin1]{inputenc}
\usepackage[active]{srcltx}
\usepackage[mathscr]{euscript}
\usepackage{enumerate}
\usepackage{stmaryrd}
\usepackage[all]{xypic}
\usepackage{xcolor}

\DeclareSymbolFont{rsfscript}{OMS}{rsfs}{m}{n}
\DeclareSymbolFontAlphabet{\mathrsfs}{rsfscript}

\newcommand{\Cl}[1]{\ensuremath{\mathcal{#1}}}
\newcommand{\unif}[1]{\ensuremath{\mathrsfs{#1}}}
\newcommand{\pv}[1]{\ensuremath{\mathsf{#1}}}
\newcommand{\compl}[2]{\ensuremath{C_{\pv{#1}}(#2)}}
\newcommand{\apex}{\ensuremath{\mathscr J}}
\newcommand\op{\llbracket}
\newcommand\cl{\rrbracket}
\newcommand{\End}[1]{\ensuremath{\mathrm{End}(#1)}}
\newcommand{\Aut}[1]{\ensuremath{\mathrm{Aut}(#1)}}
\newcommand\omt[2]{\ensuremath{\Omega_{#1}^\tau{\pv{#2}}}}
\newcommand\omc[2]{\ensuremath{\Omega_{#1}^\kappa{\pv{#2}}}}
\newcommand\clos[2]{\ensuremath{\mathrm{cl}_{\pv{#1}}(#2)}}

\newcommand\om[2]{\ensuremath{\Omega_{#1}\pv{#2}}}
\newcommand\Om[2]{\ensuremath{\overline\Omega_{#1}\pv{#2}}}
\newcommand{\shi}[1]{\ensuremath{{{#1}}^{\Z}}}
\newcommand{\ci}[1]{\ensuremath{{\mathcal {#1}}}}
\newcommand\HH{\ensuremath{\pv {\overline{H}}}}
\newcommand\GG{\ensuremath{\pv {\overline{G}}}}
\newcommand\z[1]{\ensuremath{#1^\mathbb{Z}}}

\newcommand\malcev{\mathop{\raise1pt\hbox{\footnotesize$\bigcirc$%
\kern-7pt\raise1pt\hbox{\tiny$m$}\kern1pt}}}

\newcommand\ov{\overline}
\newcommand\B{\mathcal{B}}
\newcommand\J{\mathcal{J}}

\makeindex

\begin{document}

\title{Profinite topologies}

\author{Jorge Almeida$^{1,}$\thanks{
    Work partially supported by CMUP (UID/MAT/00144/2013), 
    which is funded by FCT (Portugal) with national (MCTES) and European 
    structural funds (FEDER), under the partnership 
    agreement PT2020.}%
  \ and Alfredo Costa$^{2,}$\thanks{
    Work partially supported by the Centre for Mathematics of the
    University of Coimbra -- UID/MAT/00324/2013, funded by the Portuguese
    Government through FCT/MCTES and co-funded by the European Regional 
    Development Fund through the Partnership Agreement PT2020.} }

\markboth{J. Almeida, A. Costa}{Profinite topologies}

\address{ $^1$CMUP, Fac. Ci\^encias, Univ.
  Porto, Rua do Campo Alegre 687, 4169-007 Porto, Portugal\\
  email:\,\url{jalmeida@fc.up.pt}\\[4pt]
  $^2$CMUC, University of Coimbra, Apartado 3008, 3001-501 Coimbra,
  Portugal\\
  email:\,\url{amgc@mat.uc.pt}
}


\maketitle\label{chapter17}


\begin{abstract}
  Profinite semigroups are a generalization of finite semigroups that
  come about naturally when one is interested in considering free
  structures with respect to classes of finite semigroups. They also
  appear naturally through dualization of Boolean algebras of regular
  languages. The additional structure is given by a compact
  zero-dimensional topology. Profinite topologies may also be
  considered on arbitrary abstract semigroups by taking the initial
  topology for homomorphisms into finite semigroups. This text is the
  proposed chapter of the Handdbook of Automata Theory
  dedicated to these topics. The general theory is formulated in the
  setting of universal algebra because it is mostly independent of
  specific properties of semigroups and more general algebras
  naturally appear in this context. In the case of semigroups,
  particular attention is devoted to solvability of systems of
  equations with respect to a pseudovariety, which is relevant for
  solving membership problems for pseudovarieties. Focus is also given
  to relatively free profinite semigroups per se,
  specially ``large'' ones,
  stressing connections with symbolic dynamics that bring light to their structure.
\end{abstract}



\localtableofcontents

%
%

\section{Introduction}
\label{sec:ALCO-intro}

Profinite semigroups and profinite topologies in semigroups have
become an important tool in the theory of finite automata. There are
many reasons for this fact.

First, since finite automata describe transition finite semigroups (of
transformations or relations) by giving the action of their generators
on a finite set of states, the separation power of words by a class
\Cl C of finite automata translates in algebraic terms to the
separation power of homomorphisms from free semigroups into the
corresponding transition semigroups. More generally, the homomorphisms
from a semigroup $S$ into such semigroups determine an initial
topology on~$S$, namely the corresponding profinite topology. The
topological separation axiom of Hausdorff is the familiar algebraic
property of being residually in \Cl C. But, actually, homomorphisms
into finite semigroups give a finer structure, namely a uniform
structure, or even a metric structure in case the semigroup $S$ is
finitely generated. Thus, there is a natural completion associated
with our separation scheme, which is called the pro-\Cl C completion
of $S$. The topological semigroups thus obtained are so-called
\emph{pro-\Cl C semigroups}. For the class \Cl C of all finite
semigroups, the attribute ``pro-\Cl C'' becomes simply ``profinite''.

Another explanation for the importance of profinite topologies comes
from duality. Applying the above recipe to the free semigroup $A^+$,
with \Cl C a pseudovariety \pv V of finite semigroups, the resulting
pro-\pv V semigroup is known as the free pro-\pv V semigroup, for
indeed it has the expected universal property. It turns out that the
topological structure of this free pro-\pv V semigroup is precisely
the Stone dual of the Boolean algebra of regular languages over the
alphabet $A$ that can be recognized by members of~\pv V. The further
dualization of residual operations determines the
multiplication~\cite{Gehrke&Grigorieff&Pin:2008,Gehrke&Grigorieff&Pin:2010}.

Another fundamental reason why free profinite semigroups are important
is that their elements, sometimes called \emph{pseudowords}%
\index{pseudoword}, %
play the role of terms in classical universal algebra. Indeed,
pseudovarieties can be defined by formal equalities between
pseudowords.

To be able to apply these connections with the profinite world, some
knowledge of the structure of free pro-\pv V semigroups is usually
necessary for suitable pseudovarieties \pv V. The thus motivated
structural investigation of these semigroups is in general quite hard
and has only been carried out in a very limited number of cases.

Another major difficulty lies in the fact that in most interesting
cases, free pro-\pv V semigroups are uncountable. Thus, there are delicate
questions when trying to obtain decidability results using
pseudowords. An important idea in this context is to replace arbitrary
pseudowords by those of a special kind, namely the elements of the
subalgebra with respect to a suitably enriched language. This leads to
the notions of reducibility and tameness which are involved in some of
the deepest results using profinite methods.

The aim of this chapter is to efficiently introduce these topics,
illustrating with examples and results the wide range of application
of profinite methods.
We introduce profinite topologies in the context of general algebraic
structures. Although they were originally considered in this context
by Birkhoff~\cite{Birkhoff:1937}, so far they have not been much
studied outside the realm of group and semigroup theories. In the
context of ring theory, there is an analog topology, which may or may
not be profinite, and which is known as the \emph{Krull topology}. It
is determined on a ring by a filtration by ideals. For instance, for
the ring of $p$-adic integers, the filtration consists of the ideals
generated by the powers of the prime~$p$ and the topology is
``profinite'' in the sense that the quotient rings
$\mathbb{Z}/p^n\mathbb{Z}$ are finite.

Since most of the theory is independent of the concrete algebraic
structures in which one may be interested, and a lot of attention has
been given to general algebraic structures as recognizing devices for
tree languages (see Chapter~22), it seems
worthwhile to formulate the theory in the more general context.
Moreover, the reducibility and tameness properties involve themselves
general algebraic structures, even when semigroups are the aim of the
investigations. In Section~\ref{sec:ALCO-profinite-top-algebras},
results are formulated in the context of general algebras.
Section~\ref{sec:ALCO-sgps} deals with applications in the special
case of semigroups. Section~\ref{sec:ALCO-structure} introduces recent
results concerning the structure of free profinite semigroups over
large pseudovarieties, where connections with symbolic dynamics play
an important role.

\section{Profinite topologies for general algebras}
\label{sec:ALCO-profinite-top-algebras}

This section introduces profinite topologies for general (topological
abstract) algebras. The treatment presented here is meant to be a
quick guide to the main general results in this area. For most proofs,
the reader is referred to the bibliography. Occasionally, simple
generalizations of the previously published results are presented here
for we believe this contributes to understanding the theory, and may
be helpful in applications.

\subsection{General algebraic structures}
\label{sec:ALCO-general-algebras}

This subsection introduces the basics of Universal Algebra. The reader
is referred to \cite{Burris&Sankappanavar:1981} for further details.

By an \emph{algebraic signature}%
\index{algebraic signature}%
\index{signature!algebraic} %
we mean a set $\sigma$, of \emph{operation symbols}, together with an
\emph{arity} function $\nu:\sigma\to\mathbb{N}$ into the set of
non-negative integers. We denote $\nu^{-1}(n)$ by $\sigma_n$. A
\emph{$\sigma$-algebra}%
\index{$\sigma$-algebra}%
\index{algebra!$\sigma$--} %
consists of a nonempty set $S$ together with an interpretation
function assigning to each operation symbol $f\in\sigma$ a
$\nu(f)$-ary operation $f^S:S^{\nu(f)}\to S$. The operations on $S$ of
this form are called the \emph{basic operations}. Usually, the
interpretation function is understood and we talk about the
\emph{algebra}~$S$. Moreover, unless explicit mention of the signature
$\sigma$ is required, which is usually understood from the context, we
will omit reference to it. An algebra $S$ is \emph{trivial}%
\index{algebra!trivial} %
if $S$ is a singleton.

From hereon, whenever we talk about algebras and their classes, unless
otherwise stated, we always assume that the same signature is
involved.

A \emph{homomorphism}%
\index{algebra!homomorphism} %
is a mapping $\varphi:S\to T$ between two algebras such that, for
every arity $n$ and $f\in\sigma_n$, and for all $s_1,\ldots,s_n\in S$,
the equality $\varphi\bigl(f^S(s_1,\ldots,s_n)\bigr)
=f^T\bigl(\varphi(s_1),\ldots,\varphi(s_n)\bigr)$ holds.
For an algebra $T$, a nonempty subset $S$ closed under the
interpretation in~$T$ of the operation symbols is an algebra under the
induced operations; we then say that $S$ is a \emph{subalgebra}%
\index{subalgebra}%
\index{algebra!sub--} %
of~$T$. For a family $(S_i)_{i\in I}$ of algebras, their \emph{direct
  product}\index{direct product}%
\index{product!direct} %
$\prod_{i\in I}S_i$ is the Cartesian product with operation symbols
interpreted component-wise. Note that, if $I=\emptyset$, then
$\prod_{i\in I}S_i$ is a trivial algebra.

A \emph{congruence}%
\index{congruence}%
\index{algebra!congruence} %
on an algebra $S$ is an equivalence relation $\theta$ on~$S$ such that
$\theta$ is a subalgebra of~$S\times S$. For a congruence $\theta$
on~$S$, we may interpret each operation symbol $f\in\sigma_n$, on the
quotient set $S/\theta$ by putting
$f^{S/\theta}(s_1/\theta,\ldots,s_n/\theta)
=f^{S}(s_1,\ldots,s_n)/\theta$, whenever $s_1,\ldots,s_n\in S$, where
$s/\theta$ denotes the $\theta$-class of~$s$; this is called the
\emph{quotient algebra}%
\index{algebra!quotient} %
of~$S$ by~$\theta$. The chosen structure of~$S/\theta$ is the unique
way of defining the quotient algebra so that the natural mapping $S\to
S/\theta$, which sends $s$ to $s/\theta$, is a homomorphism.

Given an algebra $S$ and a nonempty family $(\theta_i)_{i\in I}$ of
congruences on~$S$, there is a natural injective homomorphism
$S/(\bigcap_{i\in I}\theta_i)\to\prod_{i\in I}S/\theta_i$. For a class
of algebras \Cl C containing trivial algebras, and an algebra $S$, we
denote by $\theta_\Cl C$ the intersection of the family of all
congruences $\theta$ on~$S$ such that $S/\theta\in\Cl C$. We say that
$S$~is \emph{residually in~\Cl C}%
\index{algebra!residually in} %
if $\theta_\Cl C$ is the equality relation $\Delta_S$ on the set~$S$.

A \emph{variety}%
\index{variety}%
\index{algebras!variety of --} %
is a class \Cl V of algebras which is closed under taking
homomorphic images, subalgebras and arbitrary direct products. Since
the intersection of a nonempty family of varieties is again a variety,
we may consider the \emph{variety generated by}%
\index{variety!generated by} %
any class \Cl C of algebras, denoted $\Cl V(\Cl C)$. By a well-known
theorem of Birkhoff~\cite{Birkhoff:1935}, for every nonempty set $A$
and every variety~\Cl V, there is a \emph{\Cl V-free algebra on~$A$}%
\index{algebra!free --}%
\index{algebra!relatively free --}, %
that is an algebra $F_A\Cl V$%
\index{$F_A\Cl V$} %
together with a mapping $\iota:A\to F_A\Cl V$ such that, for every
mapping $\varphi:A\to S$ into an algebra $S$ from~\Cl V, there is a
unique homomorphism $\hat\varphi:F_A\Cl V\to S$ such that
$\hat\varphi\circ\iota=\varphi$. By the usual `abstract nonsense',
such an algebra is unique up to isomorphism and depends only on the
variety \Cl V and the cardinality of the set~$A$. In case $A$~is a
finite set of cardinality $n$, we may write $F_n\Cl V$%
\index{$F_n\Cl V$} %
instead of $F_A\Cl V$. A similar convention applies for other
notations for free algebras that are used in this chapter.

In particular, the class of all $\sigma$-algebras is a variety. The
corresponding free algebra on~$A$ is the algebra $T_A^{(\sigma)}$%
\index{$T_A^{(\sigma)}$} %
of formal \emph{($\sigma$-)terms}%
\index{algebra!term}%
\index{term}, %
constructed recursively from the elements of~$A$ by formally applying
the operation symbols, which also defines their interpretation:
\begin{itemize}
\item for each $a\in A$, we have $a\in T_A$;
\item if each $t_1,\ldots,t_n$ is in~$T_A$ and $f\in\sigma_n$, then
  $f(t_1,\ldots,t_n)$ is also in~$T_A$;
\item all elements of $T_A$ are obtained by applying the preceding
  rules.
\end{itemize}
In fact the algebra $F_A\Cl V$ is naturally constructed as the
quotient algebra $T_A/\theta_\Cl V$.

For a variety \Cl V, each element $w$ of~$F_A\Cl V$ determines a
function $w_S:S^A\to S$ on each algebra $S$ from~\Cl V by letting
$w_S(\varphi)=\hat\varphi(w)$ for a function $\varphi:A\to S$.
In case
$A=\{a_1,\ldots,a_n\}$, one may prefer to view $w_S$ as a function
from $S^n$ to~$S$, by putting $w_S(s_1,\ldots,s_n)=\hat\varphi(w)$,
where $\varphi:A\to S$ maps $a_i$ to $s_i$
($i=1,\ldots,n$).

An \emph{identity}%
\index{identity} %
is a formal equality $u=v$ with $u,v\in T_A$ for some set~$A$. We say
that an algebra $S$ satisfies the identity $u=v$ if $u_S=v_S$. For a
set $\Sigma$ of identities, the class $[\Sigma]$%
\index{$[\Sigma]$} %
consisting of all algebras that satisfy all the identities
from~$\Sigma$ is easily seen to be a variety. Birkhoff's variety
theorem \cite{Birkhoff:1935} states that every variety is of this
form.

A \emph{pseudovariety}%
\index{pseudovariety}%
\index{algebras!pseudovariety of --} %
is a nonempty class \pv V of finite algebras that is closed under
taking homomorphic images, subalgebras and finite direct products. The
\emph{pseudovariety generated by}%
\index{pseudovariety!generated by}%
\index{algebras!pseudovariety of -- generated by} %
a class \Cl C of finite algebras, denoted $\pv V(\Cl
C)$, is the intersection of all pseudovarieties that contain~\Cl C. A
class of finite algebras closed under taking isomorphic algebras,
subalgebras, and finite direct products is called a
\emph{pseudoquasivariety}%
\index{pseudoquasivariety}%
\index{algebras!pseudoquasivariety of --}.

\begin{example}\label{eg:ALCO-algebras}
(1)  For the signature consisting of a single binary operation, the class
  \Cl S of all semigroups is a variety, defined by the identity
  $x(yz)=(xy)z$. Its free algebra $F_A\Cl S$ is the semigroup of words
  $A^+$. The class \pv S of all finite semigroups is a pseudovariety.
  The classes \Cl M, of all monoids, and \Cl G, of all groups, are not
  varieties: they are not closed under taking subalgebras. The class
  \pv G of all finite groups is a pseudovariety, but the class \pv M
  of all finite monoids is not a pseudovariety.

(2)  For the signature consisting of a
  binary and a nullary operation (or \emph{constant}), the class \Cl
  M is a variety and the class \pv M is a pseudovariety.

(3)  For the signature consisting of a
  binary operation, a unary operation and a nullary operation, the
  class \Cl G of all groups is a variety.

  (4) For the signature of~(1), consider the class of all finite
  semigroups such that, if an element $s$ generates a subsemigroup
  whose subgroups are trivial, then $s^2=s$. This is a
  pseudoquasivariety but not a pseudovariety (for instance, the
  3-element semigroup with presentation $\langle a: a^4=a^2\rangle$
  belongs to the class but its quotient $\langle a: a^3=a^2\rangle$
  does not).
\end{example}

Given an algebra $S$ and a subset $L$ of~$S$, the \emph{syntactic
  congruence}%
\index{congruence!syntactic}%
\index{algebra!syntactic congruence} %
of~$L$ on~$S$ is the largest congruence $\sim_L$ such that $L$ is a
union of $\sim_L$-classes. It is characterized by the following
property: for $s,s'\in S$, the relation $s\mathrel{\sim_L}s'$ holds
if and only if, for all $n\ge1$, $t\in T_n$, and $s_2,\ldots,s_n\in
S$, we have $t_S(s,s_2,\ldots,s_n)\in L$ if and only if
$t_S(s',s_2,\ldots,s_n)\in L$. For some varieties, such as of
semigroups, monoids, groups, or rings, and for any finitely generated
variety of lattices, it turns out that, rather than considering all
terms in the preceding equivalence, it suffices to consider a finite
number of them. For instance, for the variety of monoids, it suffices
to consider the single term $t=(xy)z$, as in the usual definition of
the syntactic congruence for monoids. See Clark \textit{et
  al.}~\cite{Clark&Davey&Freese&Jackson:2004} for alternative
characterizations of varieties with such a finiteness property.

For an algebra $S$, we say that a subset $L$ of~$S$ is
\emph{recognized}%
\index{subset!recognized by homomorphism} %
by a homomorphism $\varphi:S\to T$ if $L=\varphi^{-1}\varphi L$. In
other words, $L$~is a union of classes of the \emph{kernel}%
\index{congruence!kernel}%
\index{kernel} %
congruence $\ker\varphi=(\varphi\times\varphi)^{-1}\Delta_S$ or,
equivalently, $\ker\varphi$ is contained in $\sim_L$. For a class
\Cl C of algebras, we say that a subset $L$ of~$S$ is \emph{\Cl
  C-recognizable}%
\index{subset!recognizable} %
if $L$~is recognized by a homomorphism $\varphi:S\to T$ into some
algebra $T$ from~\Cl C. In particular $L$ is recognizable by some
finite algebra if and only if $\sim_L$~has finite index, in which
case we also say simply that $L$~is \emph{recognizable}.

\subsection{Pseudometric and uniform spaces}
\label{sec:ALCO-unif-spaces}

A \emph{pseudometric}%
\index{pseudometric} %
on a set $X$ is a function $d$ from $X\times X$ to the non-negative
reals such that the following conditions hold:
\begin{conditionsiii}
\item $d(x,x)=0$ for every $x\in X$;
\item $d(x,y)=d(y,x)$ for all $x,y\in X$;
\item (triangle inequality) $d(x,z)\le d(x,y)+d(y,z)$ for all
  $x,y,z\in X$.
\end{conditionsiii}
In case, additionally, $d(x,y)=0$ implies $x=y$, then we say that
$d$~is a \emph{metric}%
\index{metric} %
on~$X$. If, instead of the triangle inequality, we impose the stronger
\begin{conditionsiii}
\item[(iv)] (ultrametric inequality) $d(x,z)\le\max\{d(x,y),d(y,z)\}$
  for all $x,y,z\in X$,
\end{conditionsiii}
then we refer respectively to a \emph{pseudo-ultrametric}%
\index{pseudo-ultrametric} %
and an \emph{ultrametric}. For each of these types of ``something''
metrics, a \emph{``something'' metric space}%
\index{metric!space} %
is a set endowed with a ``same thing'' metric.

The remainder of this section is dedicated to recalling the notion of
a uniform space. We build up here on the approach
of~\cite{Almeida&Steinberg:2008}. The reader may prefer to consult a
book on general topology such as~\cite{Willard:1970}.

\begin{definition}\label{d:ALCO-uniformity}
  A \emph{uniformity} on a set $X$ is a set~\unif U of reflexive
  binary relations on~$X$ such that the following conditions hold:
  \begin{enumerate}
  \item\label{item:ALCO-uniformity-1} if $R_1\in\unif U$ and
    $R_1\subseteq R_2$, then $R_2\in\unif U$;
  \item\label{item:ALCO-uniformity-2} if $R_1,R_2\in\unif U$,
    then there exists $R_3\in\unif U$ such that $R_3\subseteq R_1\cap
    R_2$;
  \item\label{item:ALCO-uniformity-3} if $R\in\unif U$, then
    there exists $R'\in\unif U$ such that $R'\circ R'\subseteq R$;
  \item\label{item:ALCO-uniformity-4} if $R\in\unif U$, then $R^{-1}\in\unif U$.
  \end{enumerate}
  An element of a uniformity is called an \emph{entourage}%
  \index{entourage}. %
  A \emph{uniform space}%
  \index{uniform space}%
  \index{space!uniform} %
  is a set endowed with a uniformity, which is usually understood and
  not mentioned explicitly.

  A \emph{uniformity basis}%
  \index{uniformity!basis}%
  \index{basis!uniformity} %
  on a set~$X$ is a set \unif U of reflexive binary relations on~$X$
  satisfying the above conditions
  (\ref{item:ALCO-uniformity-2})--(\ref{item:ALCO-uniformity-4}). The
  \emph{uniformity generated}%
  \index{uniformity!generated by} %
  by~\unif U consists of all binary relations on~$X$ that contain some
  member of~\unif U.

  A uniformity \unif U is \emph{transitive}%
  \index{uniformity!transitive --} %
  if it admits a basis consisting of transitive relations.
\end{definition}

The notion of a uniform space generalizes that of a pseudometric
space. In this respect, the following notation is suggestive of the
intuition behind the generalization. For an entourage $R$ and elements
$x,y\in X$, we write $d(x,y)<R$ to indicate that $(x,y)\in R$.
Indeed, given a metric $d$ on~$X$, if we let $R_\epsilon$ denote the
set of pairs $(x,y)\in X\times X$ such that $d(x,y)<\epsilon$, then
the set $\unif U_d$ of all $R_\epsilon$, with $\epsilon>0$, is a
uniformity basis on~$X$ such that $d(x,y)<R_\epsilon$ if and only if
$d(x,y)<\epsilon$. The uniformity $\unif U_d$ is said to be
\emph{defined} by~$d$.

The \emph{topology of}%
\index{uniform space!topology of --} %
a uniform space $X$ (or \emph{induced by}%
\index{uniformity!topology induced by --} %
its uniformity) has neighborhood basis for each $x\in X$
consisting of all sets of the form $B_R(x)=\{y\in X:d(x,y)<R\}$. Not
every topology is induced by a uniformity
\cite[Theorem~38.2]{Willard:1970}.

Note that the topology induced by a uniformity \unif U on~$X$ is Hausdorff if
and only if the intersection $\bigcap\unif U$ is the
diagonal (equality) relation $\Delta_X$. In general, it follows from
the definition of uniformity that $\bigcap\unif U$ is an equivalence
relation on~$X$. The quotient set $X/\bigcap\unif U$ is then naturally
endowed with the \emph{quotient uniformity}%
\index{uniformity!quotient --}, %
whose entourages are the relations $R/\bigcap\unif U$, with $R\in\unif
U$. Of course, the \emph{quotient space} $X/\bigcap\unif U$ is
Hausdorff and we call it the \emph{Hausdorffization}%
\index{Hausdorffization} %
of~$X$ while the natural mapping $X\to X/\bigcap\unif U$ is called the
\emph{natural Hausdorffization mapping}%
\index{Hausdorffization!natural -- mapping}. %
Given a uniformity \unif U on a set~$X$ and a subset $Y$, the
\emph{relative uniformity}%
\index{uniformity!relative --} %
on~$Y$ consists of the entourages of the form $R\cap(Y\times Y)$ with
$R\in\unif U$. Endowed with this uniformity, $Y$~is said to be a
\emph{uniform subspace}%
\index{uniform!subspace}%
\index{subspace!uniform --} %
of~$X$.

Recall that a \emph{net}%
\index{net} %
in a set $X$ is a function $f:I\to X$, where $I$ is a directed set,
meaning a set endowed with a partial order $\le$ such that, for all
$i,j\in I$, there is some $k\in I$ with $i\le k$ and $j\le k$. A
subnet of such a net is a net $g:J\to X$ for which there is an
order-preserving function $\lambda:J\to I$ such that $g=f\circ
\lambda$ and, for every $i\in I$, there is some $j\in J$ with $i\le
\lambda(j)$, that is, $\lambda$~has cofinal image in~$I$. Usually, the
net $f$ is represented by $(x_i)_{i\in I}$, where $x_i=f(i)$. The
subnet $g$ is then represented by $(x_{i_j})_{j\in J}$, where
$i_j=\lambda(j)$. In case $X$~is a topological space, we say that the
net $(x_i)_{i\in I}$ \emph{converges to}%
\index{net!converges to}%
~$x\in X$ if, for every neighborhood $N$ of~$x$, there is some $i\in
I$ such that $x_j\in N$ whenever $j\ge i$.

A net $(x_i)_{i\in I}$ in a uniform space $X$ is said to be a
\emph{Cauchy net}%
\index{net!Cauchy --} %
if, for every entourage~$R$, there is some $i\in I$ such that
$d(x_j,x_k)<R$ whenever $j,k\ge i$. A uniform space is said to be
\emph{complete}%
\index{uniform space!complete --} %
if every Cauchy net converges.

A Hausdorff topological space $X$ is said to be \emph{compact}%
\index{compact space}%
\index{topological space!compact --} %
if every open covering of $X$ contains a finite covering.
Equivalently, every net in~$X$ has a convergent subnet. A topological
space is said to be \emph{zero-dimensional}%
\index{zero-dimensional space}%
\index{topological space!zero-dimensional --} %
if it admits a basis consisting of \emph{clopen}%
\index{clopen set}%
\index{set!clopen} %
sets, that is sets that are both closed and open. It is well known
that a compact space is zero-dimensional if and only if it is
\emph{totally disconnected}%
\index{totally disconnected space}%
\index{topological space!totally disconnected --}, %
meaning that all its connected components are singleton sets. One can
also show that a compact space has a unique uniformity that induces
its topology \cite[Theorem~36.19]{Willard:1970}.

A uniform space $X$ is \emph{totally bounded}%
\index{totally bounded|see{uniform space}}%
\index{uniform space!totally bounded --} %
if, for every entourage $R$, there is a finite cover
$X=U_1\cup\cdots\cup U_n$ such that $\bigcup_{k=1}^n U_k\times
U_k\subseteq R$. It is well known that a Hausdorff uniform space is
compact if and only if it is complete and totally bounded
\cite[Theorem~39.9]{Willard:1970}.

A function $\varphi:X\to Y$ between two uniform spaces is
\emph{uniformly continuous}%
\index{function!uniformly continuous --} %
if, for every entourage $R$ of~$Y$, there is some entourage $R'$
of~$X$ such that $d(x_1,x_2)<R'$ implies
$d(\varphi(x_1),\varphi(x_2))<R$. Equivalently, $\varphi$ maps Cauchy
nets to Cauchy nets. We say that $\varphi$ is a \emph{uniform
  isomorphism}%
\index{uniform space!isomorphism} %
if it is a uniformly continuous bijection whose inverse is also
uniformly continuous. The function $\varphi$~is a \emph{uniform
  embedding}%
\index{uniform space!embedding} %
if $\varphi$~is a uniform isomorphism of~$X$ with a subspace of~$Y$.
Note that, if $\varphi:X\to Y$ is a uniformly continuous function,
then $\varphi$ induces a unique uniformly continuous function
$\psi:X/\bigcap\unif U_X\to Y/\bigcap\unif U_Y$ between the
corresponding Hausdorffizations such that
$\psi\circ\pi_X=\pi_Y\circ\varphi$, where $\pi_X$ and $\pi_Y$ are the
natural Hausdorffization mappings. We call $\psi$ the
\emph{Hausdorffization}%
\index{Hausdorffization} %
of~$\varphi$.

One can show \cite[Theorem~38.3]{Willard:1970}
that a uniformity is defined
by some pseudometric (respectively by a pseudo-ultrametric) if and
only if it has a countable basis (and, respectively, it is
    transitive). In the Hausdorff case, one can remove the prefix
``pseudo''.
Moreover, every uniform space can be uniformly embedded in
a product of pseudometric spaces \cite[Theorem~39.11]{Willard:1970}.

For every uniform space $X$ there is a complete uniform space $\hat X$
such that $X$ embeds uniformly in $\hat X$ as a dense subspace. This
can be done by first uniformly embedding $X$ in a product of
pseudometric spaces and then completing each factor by diagonally
embedding it in the space of equivalence classes of Cauchy sequences
under the relation $(x_n)_n\approx(y_n)_n$ if $\lim d(x_n,y_n)=0$
(cf.~\cite[Theorems 39.12 and 24.4]{Willard:1970}).

Such a space $\hat X$ is unique in the sense that, given any other
complete uniform space $Y$ in which $X$~embeds uniformly as a dense
subspace, there is a unique uniform isomorphism $\hat X\to Y$ leaving
$X$ pointwise fixed. The uniform space $\hat X$ is called the
\emph{completion}%
\index{uniform space!completion} %
of~$X$. It is easy to verify that the Hausdorffization of the
completion of~$X$ is the completion of the Hausdorffization of~$X$; it
is known as the \emph{Hausdorff completion}%
\index{uniform space!Hausdorff completion} %
of~$X$. Moreover, the Hausdorff completion of~$X$ is compact if and
only if $X$~is totally bounded. The following is a key property of
completions.

\begin{proposition}\label{p:ALCO-extensions-to-completions}
  Let $X$ and $Y$ be uniform spaces and let $\varphi:X\to Y$ be a
  uniformly continuous function. Then there is a unique extension
  of~$\varphi$ to a uniformly continuous function $\hat\varphi:\hat
  X\to\hat Y$.
\end{proposition}

Let $I$ be a nonempty set. If $\unif U_i$ is a uniformity on a
set~$X_i$ for each $i\in I$, then the Cartesian product $\prod_{i\in
  I}X_i$ may be endowed with the \emph{product uniformity}%
\index{uniformity!product --}, %
with basis consisting of all sets of the form
$p_{i_1}^{-1}(R_1)\cap\cdots\cap p_{i_n}^{-1}(R_n)$, where each
$R_j\in\unif U_{i_j}$ and each $p_i:X\times X\to X_i\times X_i$ is the
natural projection on each component. From the fact that a nonempty
product of complete uniform spaces is complete
\cite[Theorem~39.6]{Willard:1970}, it follows that completion and
product commute. One can also easily show that Hausdorffization and
product commute.

\subsection{Profinite uniformities and metrics}
\label{sec:ALCO-prof-unif-metr}

By a \emph{topological algebra}%
\index{algebra!topological --} %
we mean an algebra endowed with a topology with respect to which each
basic operation is continuous. A \emph{compact algebra}%
\index{algebra!compact} %
is a topological algebra whose topology is compact. We view finite
algebras as topological algebras with respect to the discrete
topology. When we write that two topological algebras are isomorphic
we mean that there is an algebraic isomorphism between them which is
also a homeomorphism. A subset $X$ of a topological algebra $S$ is
said to \emph{generate}%
\index{algebra!topological -- generated by} %
$S$ if it generates a dense subalgebra of~$S$.

Similarly, a \emph{uniform algebra}%
\index{algebra!uniform --} %
is an algebra endowed with a uniformity such that the basic operations
are uniformly continuous. Note that a uniform algebra is also a
topological algebra for the topology induced by the uniformity and
that, in case the topology is compact, the basic operations are
continuous if and only if they are uniformly continuous (for the
unique uniformity inducing the topology). Consistently with the choice
of the discrete topology for finite algebras, we endow them with the
\emph{discrete uniformity}%
\index{uniformity!discrete --}, %
in which every reflexive relation is an entourage.

Let \Cl F be a class of finite algebras. A subset $L$ of a topological
(respectively uniform) algebra $S$~is said to be \emph{\Cl
  F-recognizable} if there is a continuous (resp.~uniformly
continuous) homomorphism $\varphi:S\to P$ into some $P\in\Cl F$ such
that $L=\varphi^{-1}\varphi L$. In case \Cl F consists of all finite
algebras, we say simply that $L$ is \emph{recognizable}%
\index{subset!recognizable -- of topological algebra}%
\index{algebra!recognizable subset of topological --} %
to mean that it is \Cl F-recognizable.

Let \Cl T be a class of topological algebras. A topological algebra
$S$~is said to be \emph{residually in~\Cl T}%
\index{algebra!topological -- residually in} %
if, for every pair of distinct points $s,t\in S$, there exists a
continuous homomorphism $\varphi:S\to P$, into some $P\in\Cl T$, such
that $\varphi(s)\ne\varphi(t)$.

Suppose that $S$ is a topological algebra and \pv Q is a
pseudoquasivariety. The case that will interest us the most is when
\pv Q is a pseudovariety and $S$~is a discrete algebra. The
\emph{pro-\pv Q uniformity} on~$S$, denoted $\unif U_\pv Q$, is
generated by the basis consisting of all congruences $\theta$ such
that $S/\theta\in\pv Q$ and the natural mapping $S\to S/\theta$ is
continuous. Note that $\unif U_\pv Q$ is indeed a uniformity on~$S$,
which is transitive. In case \pv Q consists of all finite algebras, we
also call the pro-\pv Q uniformity the \emph{profinite uniformity}%
\index{uniformity!profinite --}. %
The pro-\pv Q uniformity on~$S$ is Hausdorff if and only if $S$ is
residually in~\pv Q as a topological algebra. More precisely, the
Hausdorffization of~$S$ is given by the pro-\pv Q uniform structure of
$S/\theta_\pv Q$, under the quotient topology. The topology induced by
the pro-\pv Q uniformity of the algebra $S$ is also called its
\emph{pro-\pv Q topology}%
\index{topology!pro-$\mathsf{Q}$ --}%
\index{topology!profinite --}. %
Sets that are open in this topology are also said to be \emph{\pv
  Q-open} and a similar terminology is adopted for closed and clopen
sets. Similar notions can be defined if we start with a uniform
algebra instead of a topological algebra, replacing continuity by
uniform continuity, but we will have no use for them here.

Note that the pro-\pv Q uniformity $\unif U_\pv Q$ is totally bounded
for a pseudoquasivariety \pv Q. Given
a subset $L$ of an algebra $S$, we denote by $E_L$ the equivalence
relation whose classes are $L$ and its complement $S\setminus L$. Note
that, for a congruence $\theta$ on~$S$, we have $\theta=\bigcap_L
E_L$, where the intersection runs over all $\theta$-classes.
The following is now immediate.

\begin{proposition}\label{p:ALCO-pro-Q-uniformity}
  Suppose that \pv Q is a pseudoquasivariety and $S$~is a topological
  algebra.
  \begin{enumerate}
  \item\label{item:ALCO-pro-Q-uniformity-1} The Hausdorff completion
    of $S$ under $\unif U_\pv Q$ is compact.
  \item\label{item:ALCO-pro-Q-uniformity-2} A subset $L$ of~$S$ is \pv
    Q-recognizable if and only if $E_L$ belongs to~$\unif U_\pv Q$. In
    case \pv Q is a pseudovariety, a further equivalent condition is
    that the syntactic congruence $\sim_L$ belong to~$\unif U_\pv Q$.
  \item\label{item:ALCO-pro-Q-uniformity-3} The \pv Q-recognizable
    subsets of~$S$ are \pv Q-clopen and constitute a basis of the
    pro-\pv Q topology of~$S$. In particular, the pro-\pv Q topology
    of~$S$ is zero-dimensional and a subset $L$ of~$S$ is \pv Q-open
    if and only if $L$ is a union of \pv Q-recognizable sets.
  \end{enumerate}
\end{proposition}

In contrast, not every \pv Q-clopen subset of an algebra $S$ needs to
be \pv Q-recognizable. For instance, for the pseudovariety \pv N, of
all finite nilpotent semigroups, one may easily show that the pro-\pv
N topology on the (discrete) free semigroup $A^+$ over a finite
alphabet $A$ is discrete, and so every subset is clopen, while it is
well-known that the \pv N-recognizable subsets of~$A^+$ are the finite
and cofinite languages.

For a pseudoquasivariety \pv Q and a topological algebra $S$, we
define two functions on $S\times S$ as follows. For $s,t\in S$, $r_\pv
Q(s,t)$~is the minimum of the cardinalities of algebras $P$ from~\pv Q
for which there is some continuous homomorphism $\varphi:S\to P$ such
that $\varphi(s)\ne\varphi(t)$, where we set $\min\emptyset=\infty$.
We then put $d_\pv Q(s,t)=2^{-r_\pv Q(s,t)}$ with the convention that
$2^{-\infty}=0$. One can easily check that $d_\pv Q$ is a
pseudo-ultrametric on~$S$, which is called the \emph{pro-\pv Q
  pseudo-ultrametric}%
\index{pseudo-ultrametric!pro-$\mathsf{Q}$ --} %
on~$S$.

The following result is an immediate generalization 
of~\cite[Section~3]{Pin&Silva:2011},
where the hypothesis that the
signature is finite serves to guarantee that there are at most
countably many isomorphism classes of finite $\sigma$-algebras.

\begin{proposition}\label{p:ALCO-uniformity-vs-ultrametric}
  Suppose that $\sigma$~is a finite signature. For a
  pseudoquasivariety \pv Q and a topological algebra $S$, the
  following conditions are equivalent:
  \begin{enumerate}
  \item the pro-\pv Q uniformity on~$S$ is defined by the pro-\pv Q
    pseudo-ultrametric on~$S$;
  \item the pro-\pv Q uniformity on~$S$ is defined by some
    pseudo-ultrametric on~$S$;
  \item there are at most countably many \pv Q-recognizable subsets
    of~$S$;
  \item for every $P\in\pv Q$, there are at most countably
    many homomorphisms $S\to P$.    
  \end{enumerate}
  In particular, all these conditions hold in case $S$~is finitely
  generated. Moreover, if \pv Q contains nontrivial algebras then, for
  the discrete free algebra $F_A\Cl Q$ over the variety generated
  by~\pv Q, the pro-\pv Q uniformity is defined by the pro-\pv Q
  pseudo-ultrametric if and only if $A$~is finite.
\end{proposition}

The next result gives a different way of looking into pro-\pv Q
topologies and uniformities.

\begin{proposition}\label{p:ALCO-pro-Q-uniformity-vs-algebra}
  Let $S$ be a topological algebra and \pv Q a pseudoquasivariety.
  \begin{enumerate}
  \item\label{i:ALCO-pro-Q-uniformity-vs-algebra-1} The pro-\pv Q
    uniformity of~$S$ is the smallest uniformity \unif U on~$S$ for
    which all continuous homomorphisms from $S$ into members of~$Q$
    are uniformly continuous.
  \item\label{i:ALCO-pro-Q-uniformity-vs-algebra-2} The pro-\pv Q
    topology of~$S$ is the smallest topology \unif T on~$S$ for which
    all continuous homomorphisms from $S$ into members of~$Q$ remain
    continuous.
  \item\label{i:ALCO-pro-Q-uniformity-vs-algebra-3} The algebra $S$ is
    a uniform algebra with respect to its pro-\pv Q uniformity. In
    particular, it is a topological algebra for its pro-\pv Q
    topology.
  \end{enumerate}
\end{proposition}

Following \cite{Pin&Silva:2011}, we say that a function $\varphi:S\to
T$ between two topological algebras is \emph{$(\pv Q,\pv R)$-uniformly
  continuous} if it is uniformly continuous with respect to the
uniformities $\unif U_\pv Q$, on~$S$, and $\unif U_\pv R$, on~$T$.
Similarly, we say that $\varphi$~is \emph{$(\pv Q,\pv R)$-continuous}
if it is continuous with respect to the \pv Q-topology of~$S$ and the
\pv R-topology of~$T$.

It is now easy to deduce the following result, which is a
straightforward generalization of~\cite[Theorem~4.1]{Pin&Silva:2011}.

\begin{proposition}\label{p:ALCO-continuity}
  Let \pv Q and \pv R be two pseudoquasivarieties, $S$ and $T$ be two
  topological algebras, and $\varphi:S\to T$ an arbitrary function.
  \begin{enumerate}
  \item The function $\varphi$~is $(\pv Q,\pv R)$-uniformly continuous
    if and only if, for every \pv R-recogniz\-able subset $L$
    of\/~$T$, $\varphi^{-1}L$ is a \pv Q-recognizable subset of~$S$.
  \item The function $\varphi$~is $(\pv Q,\pv R)$-continuous if
    and only if, for every \pv R-recognizable subset $L$ of\/~$T$,
    $\varphi^{-1}L$ is a union of \pv Q-recognizable subsets of~$S$.
  \end{enumerate}
\end{proposition}

Proposition~\ref{p:ALCO-continuity} was motivated by the work of Pin
and Silva~\cite{Pin&Silva:2014} on non-commutative versions of
Mahler's theorem in $p$-adic Number Theory, which states that a
function $\mathbb{N}\to\mathbb{Z}$ is uniformly continuous with
respect to the $p$-adic metric if and only if it can be uniformly
approximated by polynomial functions.

\subsection{Profinite algebras}
\label{sec:ALCO-profinite-algebras}

This subsection is mostly based on \cite{Almeida:2002a}, where the
reader may find further details.

For a class \Cl T of topological algebras, a \emph{pro-\Cl T algebra}%
\index{algebra!pro-${\cal T}$ --} %
is a compact algebra that is residually in~\Cl T. A \emph{profinite
  algebra}%
\index{algebra!profinite --} %
is a pro-\Cl T algebra where \Cl T is the class of all finite algebras.

An \emph{inverse system} $\Cl I=(I,S_i,\varphi_{ij})$ of topological
algebras consists of a family $(S_i)_{i\in I}$ of such algebras,
indexed by a directed set $I$, together with a family
$(\varphi_{ij})_{i,j\in I; i\ge j}$ of functions,
the \emph{connecting homomorphisms}, such that the
following conditions hold:
\begin{conditionsiii}
\item each $\varphi_{ij}$ is a continuous homomorphism $S_i\to S_j$;
\item each $\varphi_{ii}$ is the identity function on~$S_i$;
\item for all $i,j,k\in I$ such that $i\ge j\ge k$, the equality
  $\varphi_{jk}\circ\varphi_{ij}=\varphi_{ik}$ holds.
\end{conditionsiii}

The \emph{inverse limit} of an inverse system $\Cl
I=(I,S_i,\varphi_{ij})$ is the subspace $\varprojlim\Cl I$ of
$\prod_{i\in I}S_i$ consisting of the families $(s_i)_{i\in I}$ such
that $\varphi_{ij}(s_i)=s_j$ whenever $i\ge j$. Note that, in case
$\varprojlim\Cl I$ is nonempty, it is a subalgebra of~$\prod_{i\in
  I}S_i$ and, therefore, a topological algebra. The inverse limit may
be empty. For instance, the inverse limit of the inverse system
$(\mathbb{N},{[}n,+\infty{[},\varphi_{nm})$ is empty, where the
intervals are viewed as semilattices under the usual ordering and with
the inclusion mappings as connecting homomorphisms $\varphi_{nm}$. In
contrast, if all the $S_i$ are compact algebras, then so
is~$\varprojlim\Cl I$ \cite[Exercise~29C]{Willard:1970}.

The following is a key property of pro-\pv V algebras for a
pseudovariety~\pv V.

\begin{proposition}
  \label{c:ALCO-finite-images-of-pro-V}
  Let \pv V be a pseudovariety, $S$ a pro-\pv V algebra, and
  $\varphi:S\to T$ a continuous homomorphism onto a finite algebra.
  Then $T$ belongs to~\pv V.
\end{proposition}

More generally, for a pseudoquasivariety \pv Q, the following
alternative characterizations of pro-\pv Q algebras are straighforward
extensions of the pseudovariety case for semigroups, which can be
found, for instance, in~\cite[Proposition~4.3]{Almeida:2002a}.

\begin{proposition}\label{p:pro-Q-algebras}
  Let \pv Q be a pseudoquasivariety. Then the class $\bar{\pv Q}$ of
  all pro-\pv Q algebras consists of all inverse limits of algebras
  from~\pv Q and it is the smallest class of topological algebras
  containing \pv Q that is closed under taking isomorphic algebras,
  closed subalgebras, and arbitrary direct products. The classes
  $\bar{\pv Q}$ and \pv Q have the same finite members. In case \pv Q
  is a pseudovariety, the class $\bar{\pv Q}$~is additionally closed
  under taking profinite continuous homomorphic images.
\end{proposition}

Since every compact metric space is a continuous image of the Cantor
set~\cite[Theorem~30.7]{Willard:1970}, the profiniteness assumption in
the second part of Proposition~\ref{p:pro-Q-algebras} cannot be
dropped.

The nontrivial parts of the next theorem were first observed
in~\cite{Almeida&Weil:1994} to follow from the arguments
in~\cite{Almeida:1989b}, which in turn extend the case of semigroups,
due to Numakura~\cite{Numakura:1957}, through the approach of
Hunter~\cite{Hunter:1988}. The key ingredient is the following lemma,
first stated explicitly and proved by
Hunter~\cite[Lemma~4]{Hunter:1988} for semigroups although, in this
case, it can also be extracted from~\cite{Numakura:1957}.

\begin{lemma}\label{l:ALCO-Hunter}
  Let $S$ be a compact zero-dimensional algebra and let $L$ be a
  subset of~$S$ for which the syntactic congruence is determined by
  finitely many terms. Then $L$ is recognizable if and only if $L$ is
  clopen.
\end{lemma}

 The reader may wish to compare Lemma~\ref{l:ALCO-Hunter} with
 Proposition~\ref{p:ALCO-pro-Q-uniformity}\eqref{item:ALCO-pro-Q-uniformity-3}
 and the subsequent comments.

\begin{theorem}\label{t:ALCO-profiniteness}
  Let $S$ be a compact algebra and consider the following conditions:
  \begin{enumerate}
  \item\label{item:ALCO-profiniteness-1} $S$ is profinite;
  \item\label{item:ALCO-profiniteness-2} $S$ is an inverse limit of an
    inverse system of finite algebras;
  \item\label{item:ALCO-profiniteness-3} $S$ is isomorphic to a closed
    subalgebra of a direct product of finite algebras;
  \item\label{item:ALCO-profiniteness-4} $S$ is a compact
    zero-dimensional algebra.
  \end{enumerate}
  Then the implications
  $\eqref{item:ALCO-profiniteness-1}
  \Leftrightarrow
  \eqref{item:ALCO-profiniteness-2}
  \Leftrightarrow
  \eqref{item:ALCO-profiniteness-3}
  \Rightarrow
  \eqref{item:ALCO-profiniteness-4}$ always
  hold, while
  $\eqref{item:ALCO-profiniteness-4}
  \Rightarrow
  \eqref{item:ALCO-profiniteness-3}$
  also holds in case the syntactic
  congruence of~$S$ is determined by a finite number of terms.
\end{theorem}

One can find in~\cite{Clark&Davey&Freese&Jackson:2004} explicit proofs
of Lemma~\ref{l:ALCO-Hunter} and Theorem~\ref{t:ALCO-profiniteness}.
As mentioned in Section~\ref{sec:ALCO-general-algebras}, the same
paper provides characterizations of the finiteness assumption in
Theorem~\ref{t:ALCO-profiniteness}. In particular, compact
zero-dimensional semigroups, monoids, groups, rings, and lattices in
finitely generated varieties of lattices are profinite.

The finitely generated case of the following variant of
Lemma~\ref{l:ALCO-Hunter} can be found in~\cite{Almeida:1989b}.
The essential step for the proof of the general case can be found
in~\cite[Lemma~4.1]{Almeida:2002a}.

\begin{proposition}\label{p:ALCO-recognizable-vs-clopen-in-profinite}
  Let \pv Q be a pseudoquasivariety and let $S$ be a pro-\pv Q
  algebra. Then a subset $L$ of~$S$ is clopen if and only if it is \pv
  Q-recognizable, if and only if it is recognizable.
  In particular, the topology of~$S$ is the smallest topology for
  which all continuous homomorphisms from $S$ into algebras from~\pv Q
  (or, alternatively, into finite algebras) are continuous with
  respect to it. Hence, a topological algebra is a pro-\pv Q algebra
  if and only if it is compact and its topology coincides with its
  pro-\pv Q topology.
\end{proposition}

A way of constructing profinite algebras is via the Hausdorff
completion of an arbitrary topological algebra $S$ with respect to its
pro-\pv Q uniformity. We denote this completion by~\compl QS. The
next result can be easily deduced from
Propositions~\ref{p:ALCO-extensions-to-completions},
\ref{p:ALCO-pro-Q-uniformity},
and~\ref{p:ALCO-pro-Q-uniformity-vs-algebra}.

\begin{proposition}\label{p:ALCO-pro-Q-completion-is-pro-Q}
  Let $S$ be a topological algebra and \pv Q a pseudoquasivariety.
  Then \compl QS~is a pro-\pv Q algebra. Moreover, if $S$ is
  residually in~\pv Q, then the topology of~$S$ coincides with the
  induced topology as a subspace of~\compl QS.
\end{proposition}

It is important to keep in mind that the topology of a pro-\pv Q
algebra $S$ may not be its pro-\pv Q topology when $S$ is viewed as a
discrete algebra. To give an example, we introduce a pseudovariety
which is central in the theory of finite semigroups: the class \pv A%
\index{$\mathsf{A}$} %
of all finite \emph{aperiodic}%
\index{aperiodic semigroup}%
\index{semigroup!aperiodic --} %
semigroups whose subgroups are trivial.

\begin{example}\label{eg:ALCO-topology-of-pro-Q-algebra-not-pro-Q-topology}
  Let $\mathbb{N}$ be the discrete additive semigroup of natural
  numbers and consider its pro-\pv A completion \compl A{\mathbb{N}},
  which is obtained by adding one point, denote it $\infty$, which is
  such that $n+\infty=\infty+n=\infty$ and $\lim n=\infty$. Then the
  mapping that sends natural numbers to~1 and $\infty$ to~0 is a
  homomorphism into the semilattice $\{0,1\}$ which is not continuous
  for the topology of~\compl A{\mathbb{N}} but which is continuous for
  the pro-\pv A topology.
\end{example}

In contrast, it is a deep and difficult result that, for every
finitely generated profinite group, its topology coincides with its
profinite topology as a discrete group \cite{Nikolov&Segal:2003}. The
proof of this result depends on the classification of finite simple
groups.

The \pv Q-recognizable subsets of an algebra $S$ constitute a
subalgebra $\Cl P_\pv Q(S)$ of the Boolean algebra $\Cl P(S)$ of all
its subsets. On the other hand, a compact zero-dimensional space is
also known as a \emph{Boolean space}%
\index{Boolean space}. %
The two types of Boolean structures are linked through \emph{Stone
  duality} (cf.~\cite[Section~IV.4]{Burris&Sankappanavar:1981}), whose
easily described direction associates with a Boolean space its Boolean
algebra of clopen subsets; every Boolean algebra is obtained
in this way. The following result shows that the Boolean space \compl
QS and the Boolean algebra $\Cl P_\pv Q(S)$ are Stone duals. In it, we
adopt a convenient abuse of notation: for the natural mapping
$\iota:S\to\compl QS$ and a subset $K$ of~\compl QS, we write $K\cap
S$ for $\iota^{-1}K$, while, for a subset $L$ of~$S$, we write
$\overline{L}$ for the closure of $\iota L$ in~\compl QS.

\begin{theorem}\label{t:ALCO-Q-recognizability-in-Q-completion}
  Let \pv Q be a pseudoquasivariety and let $S$~be an arbitrary
  topological algebra. Then the following are equivalent for a subset
  $L$ of~$S$:
  \begin{enumerate}
  \item\label{item:ALCO-Q-recognizability-in-Q-completion-1} the set
    $L$ is \pv Q-recognizable;
  \item\label{item:ALCO-Q-recognizability-in-Q-completion-2} the set
    $L$ is of the form $K\cap S$ for some clopen subset $K$ of~\compl
    QS;
  \item\label{item:ALCO-Q-recognizability-in-Q-completion-3} the set
    $\overline{L}$ is open and $\overline{L}\cap S=L$.
  \end{enumerate}
  When the pro-\pv Q topology of~$S$ is discrete, a further
  equivalent condition is that $\overline{L}$ is open.
  Moreover, the clopen sets of the form $\overline{L}$ with $L$ a \pv
  Q-recognizable subset of~$S$ form a basis of the pro-\pv Q topology
  of~$S$.
\end{theorem}

Since \compl QS has further structure involved besides its
topology, which is the sole to intervene in Stone duality, one may ask
what further structure is reflected in the Boolean algebra. This
question has been investigated
in~\cite{Gehrke&Grigorieff&Pin:2008,Gehrke&Grigorieff&Pin:2010}, in
the context of the theory of semigroups and its connections with
regular languages.

For a topological algebra $S$, we denote by \End S%
\index{\End S} %
the monoid of
continuous endomorphisms of $S$. It can be viewed as a subspace of the
product space $S^S$, that is with the \emph{pointwise convergence
  topology}%
\index{topology!pointwise convergence --}. %
A classical alternative is the \emph{compact-open topology}
\index{topology!compact-open --}, %
for which a basis consists of all sets of the form $(K,U)$, which in
turn consist of all self maps $\varphi$ of~$S$ such that
$\varphi(K)\subseteq U$, where $K$~is compact and $U$~is open. These
two topologies on a space of self maps of $S$ in general do not
coincide. However, for finitely generated profinite algebras they
coincide on~\End S.
This was first proved by Hunter~\cite[Proposition~1]{Hunter:1983} and
rediscovered by the first author
\cite[Theorem~4.14]{Almeida:2003cshort} in the context of profinite
semigroups. Steinberg~\cite{Steinberg:2010pm} showed how this is
related with the classical theorem of Ascoli on function spaces. The
proofs extend easily to an arbitrary algebraic setting.

\begin{theorem}\label{t:ALCO-End-profinite}
  For a finitely generated profinite algebra $S$, the
  pointwise convergence and compact-open topologies coincide on~\End S
  and turn it into a profinite monoid such that the evaluation mapping
  $\End S\times S\to S$, sending $(\varphi,s)$ to~$\varphi(s)$, is
  continuous.
\end{theorem}

A further result from~\cite{Steinberg:2010pm} that extends to the
general algebraic setting is that finitely generated profinite
algebras are \emph{Hopfian}%
\index{profinite algebra!Hopfian --} %
in the sense that all continuous onto endomorphisms are automorphisms.

Denote by \Aut S%
\index{\Aut S} %
the group of units of~\End S, consisting of all continuous
automorphisms of~$S$ whose inverse is also continuous, the latter
restriction being superfluous in case $S$ is compact. From
Theorem~\ref{t:ALCO-End-profinite}, it follows that, for a finitely
generated profinite algebra $S$, \Aut S is a profinite group. In case
$S$~is a profinite group, this result as well as the Hopfian property of~$S$
are well known in group theory \cite{Ribes&Zalesskii:2000}.

\subsection{Relatively free profinite algebras}
\label{sec:ALCO-relat-free-prof}

Let \pv Q be a pseudoquasivariety. We say that a pro-\pv Q algebra $S$
is \emph{free pro-\pv Q} over a set $A$ if there is a mapping
$\iota:A\to S$ satisfying the following universal property: for every
function $\varphi:A\to T$ into a pro-\pv Q algebra, there is a unique
continuous homomorphism $\hat\varphi:S\to T$ such that
$\hat\varphi\circ\iota=\varphi$. The mapping $\iota$~is usually not
unique and it is said to be a \emph{choice of free generators}. The
following result is well known \cite{Almeida:2002a}.

\begin{proposition}\label{p:ALCO-existence-free-pro-Q}
  For every pseudoquasivariety \pv Q and every set $A$, there exists a
  free pro-\pv Q algebra over~$A$, namely the inverse limit of all
  $A$-generated algebras from~\pv Q, with connecting homomorphisms
  respecting the choice of generators. Up to isomorphism respecting
  the choice of free generators, it is unique.
\end{proposition}

We denote the free pro-\pv Q algebra over a set~$A$ by \Om AQ%
\index{$\Omega_A\mathsf{Q}$}. %
The notation is justified below.

An alternative way of constructing free pro-\pv Q algebras is through
the pro-\pv Q Hausdorff completion of free algebras.

\begin{proposition}\label{p:ALCO-free-pro-Q-through-completion}
  Let \pv Q be a pseudoquasivariety and let $A$ be a set. Let \Cl V be
  the variety generated by~\pv Q. Then the pro-\pv Q Hausdorff
  completion of the free algebra~$F_A\Cl V$ is a free pro-\pv Q
  algebra over~$A$.
\end{proposition}

Note that, by Proposition~\ref{p:ALCO-uniformity-vs-ultrametric}, if
$A$ is finite, then \Om AQ is metrizable. In contrast, the argument
presented in \cite[end of Section~3]{Almeida&Steinberg:2008} for
pseudovarieties of monoids may be extended to every nontrivial
pseudoquasivariety \pv Q to show that, if $A$~is infinite, then \Om AQ
is not metrizable.

A topological algebra $S$
is \emph{self-free}%
\index{profinite algebra!self-free --} %
\index{topological algebra!self-free --} %
with \emph{basis}%
\index{basis!of free profinite algebra} %
$A$ if $A$~is a generating subset of~$S$ such that every
mapping $A\to S$ extends uniquely to a continuous endomorphism of~$S$.

\begin{theorem}\label{t:ALCO-self-freeness}
  The following conditions are equivalent for a profinite algebra $S$:
  \begin{enumerate}
  \item\label{item:ALCO-self-freeness-1} the topological algebra $S$
    is self-free with basis $A$;
  \item\label{item:ALCO-self-freeness-2} there is a pseudoquasivariety
    \pv Q such that $S$ is isomorphic with \Om AQ;
  \item\label{item:ALCO-self-freeness-3} there is a pseudovariety \pv
    V such that $S$ is isomorphic with \Om AV.
  \end{enumerate}
\end{theorem}

\begin{proof}
  The implications $\eqref{item:ALCO-self-freeness-3}\Rightarrow
  \eqref{item:ALCO-self-freeness-2}\Rightarrow
  \eqref{item:ALCO-self-freeness-1}$ are obvious, so it remains to
  prove that $\eqref{item:ALCO-self-freeness-1}\Rightarrow
  \eqref{item:ALCO-self-freeness-3}$. Suppose that
  \eqref{item:ALCO-self-freeness-1} holds and let \pv V be the
  pseudovariety generated by all finite algebras that are continuous
  homomorphic images of~$S$. We claim that $S$ is isomorphic with~\Om
  AV.

  We first observe that, since $S$ is a profinite algebra, it is an
  inverse limit of finite algebras, which may be chosen to be
  continuous homomorphic images of~$S$. Hence $S$ is a pro-\pv V
  algebra and, therefore, there is a unique continuous homomorphism
  $\varphi:\Om AV\to S$ such that, for a choice of free generators
  $\iota:A\to\Om AV$, the composite $\varphi\circ\iota$ is the
  inclusion mapping $A\hookrightarrow S$. Since $S$~is generated
  by~$A$ as a topological algebra, the function $\varphi$~is
  surjective. It suffices to show that it is injective.

  Let $u,v$ be distinct points of~\Om AV. Since \Om AV is residually
  in~\pv V, there is some continuous homomorphism $\psi:\Om AV\to T$,
  onto some $T\in\pv V$, such that $\psi(u)\ne\psi(v)$. By the
  definition of~\pv V, there are continuous homomorphisms $\xi_i:S\to
  V_i$ ($i=1,\ldots,n$) onto finite algebras, a subalgebra $U$
  of~$\prod_{i=1}^n V_i$, and a surjective homomorphism $\rho:U\to T$.
  Since $\rho$~is surjective, there is a mapping $\eta:A\to U$ such
  that $\rho\circ\eta=\psi\circ\iota$. Let $\pi_i:\prod_{j=1}^nV_j\to
  V_i$ be the $i$th component projection. Since $\xi_i$ is surjective,
  there is a function $\mu_i:A\to S$ such that
  $\xi_i\circ\mu_i=\pi_i\circ\eta$. By self-freeness of~$S$, with
  basis~$A$, it follows that there is a continuous endomorphism
  $\hat\mu_i$ of $S$ such that $\hat\mu_i|_A=\mu_i$. Let
  $\zeta:S\to\prod_{i=1}^n V_i$ be the unique continuous homomorphism
  such that $\pi_i\circ\zeta=\xi_i\circ\hat\mu_i$ for $i=1,\ldots,n$.
  The following diagram depicts the relationships between these mappings.
  $$
  \xymatrix{
    &
    A\
    \ar[ld]_\iota
    \ar[rd]_(.6)\eta
    \ar@{^(->}[r]
    \ar@<-1pt> `u[rr] `r [rr]^{\mu_i}
    &
    S
    \ar@{-->}[d]^\zeta
    \ar[rd]_\zeta
    \ar[r]_{\hat\mu_i}
    &
    S
    \ar[rd]^{\xi_i}
    &
    \\
    \Om AV
    \ar[r]_\psi
    \ar[rru]^\varphi
    &
    T
    &
    U\ 
    \ar[l]^\rho
    \ar@{^(->}[r]
    &
    \prod_{j=1}^nV_j
    \ar[r]_(.6){\pi_i}
    &
    V_i
  }
  $$
  Note that $\pi_i\circ\zeta|_A=\xi_i\circ\mu_i=\pi_i\circ\eta$ for
  $i=1,\ldots,n$, which shows that $\zeta|_A=\eta$ and so the image
  of~$\zeta$ is contained in~$U$ and the chain of equalities
  $\rho\circ\zeta\circ\varphi\circ\iota=\rho\circ\zeta|_A
  =\rho\circ\eta=\psi\circ\iota$ holds, which yields
  $\rho\circ\zeta\circ\varphi=\psi$. Since $\psi(u)\ne\psi(v)$, we
  deduce that $\varphi(u)\ne\varphi(v)$, which establishes the claim
  that $\varphi$~is injective.
\end{proof}

Theorem~\ref{t:ALCO-self-freeness} not only gives a characterization
of relatively free profinite algebras in terms of properties that only
involve the algebras themselves, but also shows that, when talking
about such algebras, we may as well deal only with pseudovarieties.

Yet another description of relatively free profinite algebras is given
by algebras of implicit operations, which further provide a useful
viewpoint. For a class \Cl C of profinite algebras and a set $A$, an
\emph{$A$-ary implicit operation}%
\index{operation!implicit --}%
\index{implicit operation} %
$w$ on~\Cl C is a correspondence associating with each $S\in\Cl C$ a
continuous operation $w_S:S^A\to S$ such that, for every continuous
homomorphism $\varphi:S\to T$ between members of~\Cl C, the equality
$w_T(\varphi\circ f)=\varphi(w_S(f))$ holds for every $f\in S^A$. We
call $w_S$ the \emph{interpretation} of~$w$ in~$S$.

\begin{proposition}\label{p:ALCO-implicit-operations}
  Let \pv C be a class of finite algebras, let \pv V be the
  pseudovariety it generates, and let $A$ be a set. For $w\in\Om AV$
  and a pro-\pv V algebra $S$, let $\bar w_S:S^A\to S$ be defined by
  $\bar w_S(\varphi)=\hat\varphi(w)$, where $\hat\varphi$ is the
  unique continuous homomorphism $\Om AV\to S$ such that
  $\hat\varphi\circ \iota=\varphi$. Then $\bar w$ is an $A$-ary
  implicit operation on the class of all pro-\pv V algebras and every
  such operation is of this form. Moreover, the correspondence
  associating to $w$ the restriction of $\bar w$ to~\pv C is injective
  and, therefore, so is the correpondence $w\mapsto\bar w$.
\end{proposition}

Thus, we may as well identify each $w\in\Om AV$ with the implicit
operation $\bar w$ that it determines. In terms of implicit
operations, the interpretation of the basic operations is quite
transparent: for an $n$-ary operation symbol $f$, implicit operations
$w_1,\ldots,w_n\in\Om AV$, a pro-\pv V algebra $S$, and a function
$\varphi\in S^A$, we have
$$\bigl(f^{\Om AV}(w_1,\ldots,w_n)\bigr)_S(\varphi)
=f^S\bigl((w_1)_S(\varphi),\ldots,(w_n)_S(\varphi)\bigr).$$
In other words, the basic operations are interpreted pointwise.

Among the implicit operations on the class of all profinite algebras,
we have the projections $x_a$. More precisely, for a set $A$ and $a\in
A$, the $A$-ary projection on the $a$-component is interpreted in a
profinite algebra $S$ by $(x_a)_S(\varphi)=\varphi(a)$ for each
$\varphi\in S^A$. By restriction to pro-\pv V algebras, we also obtain
corresponding implicit operations, which we still denote~$x_a$. The
subalgebra of~\Om AV generated by the $x_a$ with $a\in A$ is
denoted~\om AV. Its elements are also known as \emph{$A$-ary explicit
  operations}%
\index{explicit operation}%
\index{operation!explicit --} %
on pro-\pv V algebras. From the universal property of~\Om AV, it
follows immediately that \om AV is the free algebra $F_A\Cl V$, where
\Cl V is the variety generated by~\pv V. The following result
explains the notation.

\begin{proposition}\label{p:ALCO-om_AV-dense-in-Om_AV}
  Let \pv V be a pseudovariety. Then the algebra \om AV is dense
  in~\Om AV.
\end{proposition}

The operational point of view has the advantage that pro-\pv V
algebras are automatically endowed with a stucture of profinite
algebras over any enriched signature obtained by adding implicit
operations on~\pv V. This idea is essential for
Subsection~\ref{sec:ALCO-decid-tamen}.

A formal equality $u=v$ between members of some \Om AV is said to be a
\emph{pseudoidentity for~\pv V}%
\index{pseudoidentity}%
\index{pseudoidentity|see{profinite identity}}; %
the elements of~$A$ are called the \emph{variables} of the
pseudoidentity. It is said to \emph{hold} in a pro-\pv V algebra $S$
if $u_S=v_S$. In case \pv V is the pseudovariety of all finite
algebras, we omit reference to~\pv V. For a set $\Sigma$ of
pseudoidentities for~\pv V, the class of all algebras from~\pv V that
satisfy all pseudoidentities from~$\Sigma$ is denoted $\op\Sigma\cl$%
\index{$\op\Sigma\cl$}; %
this class is said to be \emph{defined} by~$\Sigma$ and $\Sigma$~to be
a \emph{basis of pseudoidentities}%
\index{pseudoidentity!basis}%
\index{basis of pseudoidentities} %
for~it.

\begin{theorem}[Reiterman~\cite{Reiterman:1982}]\label{t:ALCO-Reiterman}
  A subclass of a pseudovariety \pv V is a pseudovariety if and only
  if it is defined by some set of pseudoidentities for~\pv V.
\end{theorem}

There are many alternative proofs of Reiterman's theorem, as well as
extensions to various generalizations of the algebras considered in
this chapter. The most relevant in the context of this handbook seems
to be the one obtained by Molchanov~\cite{Molchanov:1994} for
``pseudovarieties'' of algebras with predicates, also proved
independently by Pin and Weil~\cite{Pin&Weil:1996b}.

The interest in Reiterman's theorem stems from the fact that it
provides a language to obtain elegant descriptions of pseudovarieties.
Moreover, namely through the techniques described in the next
subsection, they sometimes lead to decidability results, even if in a
somewhat indirect way.

\subsection{Decidability and tameness}
\label{sec:ALCO-decid-tamen}

In the theory of regular word or tree languages, pseudovarieties serve
the purpose of providing an algebraic classification tool for certain
combinatorial properties. The properties that are amenable to this
approach have been identified, first by Eilenberg
\cite{Eilenberg:1976} for word languages, and later by the first
author \cite{Almeida:1990c,Almeida:1994a} and Steinby
\cite{Steinby:1992} for tree languages. By considering additional
relational structure on the algebras, further combinatorial properties
may be captured (see \cite{Pin:1995d,Polak:2004}).

Basically, in such an algebraic approach, one seeks to decide whether
a language has a certain combinatorial property by testing whether its
syntactic algebra has the corresponding algebraic property, that is,
if this algebra belongs to a certain pseudovariety. Thus, a property
of major interest that pseudovarieties may have is decidability of the
membership problem: given a finite algebra, decide whether or not it
belongs to the pseudovariety. We then simply say that the
pseudovariety is \emph{decidable}%
\index{pseudovariety!decidable --}. %

One way to establish that a pseudovariety is decidable is to prove
that it has a finite basis of pseudoidentities which are equalities
between implicit operations that can be effectively computed, so that
the pseudoidentities in the basis can be effectively checked. In fact,
for most commonly encountered implicit operations, the computation can
be done in polynomial time, in terms of the size of the algebra, and
so the verification of the basic pseudoidentities can then be done in
polynomial time.

However, many pseudovarieties of interest are not finitely based. For
instance, it is easy to see that, if a pseudovariety is generated by a
single algebra, then it is decidable, but it may not be finitely
based, an important example being the pseudovariety generated by the
syntactic monoid $B_2^1$ of the language $(ab)^*$ over the 2-letter
alphabet \cite{Perkins:1968,Sapir:1988a}. Moreover, contrary to a
conjecture proposed by the first author \cite{Almeida:1994a}, a
pseudovariety for which the membership problem is solvable in
polynomial time may not admit a finite basis of pseudoidentities
\cite{Volkov:1995}. Sapir has even shown that there is a finite
semigroup that generates such a pseudovariety
\cite[Theorem~3.53]{Kharlampovich&Sapir:1994}. It has recently been
announced by M.~Jackson that the membership problem for $\pv V(B_2^1)$
is NP-hard and so, provided $\mathrm{P}\ne \mathrm{NP}$, that problem
cannot be decided in polynomial time, which would solve
\cite[Problem~3.11]{Kharlampovich&Sapir:1994}.

Pseudovarieties are often described by (infinite) generating sets of
algebras. This comes about by applying some natural operator on other
pseudovarieties, like the join in the lattice of pseudovarieties. In
general, for any construction $C(S_1,\ldots,S_n)$ of an algebra from
given algebras $S_i$, perhaps under suitable restrictions or
additional data (like in the definition of semidirect product, where
an action of one of the factors on the other is required), one may
consider the pseudovariety $C(\pv V_1,\ldots,\pv V_n)$ generated by
all algebras of the form $C(S_1,\ldots,S_n)$ with each $S_i$ in a
given pseudovariety $\pv V_i$. The join is obtained in this way by
considering the usual direct product. Another type of operator of
interest is the following: for two pseudovarieties \pv V and \pv W,
their \emph{Mal'cev product}%
\index{product!Mal'cev --} %
$\pv V\malcev\pv W$ is the pseudovariety generated by all algebras $S$
for which there is a congruence $\theta$ such that $S/\theta$ belongs
to~\pv W, and each class which is a subalgebra belongs to~\pv V.

Since most such natural operators in the case of semigroups do not
preserve decidability
\cite{Albert&Baldinger&Rhodes:1992,Auinger&Steinberg:2001b}, it is of
interest to develop methods that, under suitable additional
assumptions on the given pseudovarieties, guarantee that the operator
produces a decidable pseudovariety. The starting point in the
profinite approach is to obtain a basis of pseudoidentities for the
resulting pseudovariety. In the context of semigroups and monoids,
bases theorems of this kind have been established for Mal'cev products
\cite{Pin&Weil:1996a} and various types of semidirect products
\cite{Almeida&Weil:1996}. Unfortunately, there is a gap in the proof
of the latter, so that the results are only known to hold under
certain additional finiteness hypotheses.\footnote{See
  \cite{Rhodes&Steinberg:2009qt} for a discussion and a general basis
  theorem, which in turn has not led to decidability results.} The
bases provided by such theorems for a binary operator $C(\pv V,\pv W)$
consist of pseudoidentities which are built from pseudoidentities
determined by~\pv V by substituting the variables by certain implicit
operations. The implicit operations that should be considered to test
membership in~$C(\pv V,\pv W)$ of a given finite $A$-generated algebra
$S$ are the solutions of certain systems of equations in~\Om AW,
determined by the operator $C$, subject to regular constraints
determined by each specific evaluation of the variables in~$S$ which
is to be tested. This approach was first introduced
in~\cite{Almeida:1996d,Almeida:1999b}, improved
in~\cite{Almeida&Steinberg:2000a}, and later extended
in~\cite{Almeida:2002a} and, independently and in a much more
systematic way, also in~\cite{Rhodes&Steinberg:2009qt}. The reader is
referred to~\cite{Almeida:1999b,Almeida&Steinberg:2000a,Almeida:2002a}
for the proofs of the results presented in this section.

We proceed to formalize the above ideas. Consider a set $\Sigma$ of
pseudoidentities, which we view as a \emph{system of equations}%
\index{equations!system of --}. %
The sides of the equations $u=v$ in~$\Sigma$ are implicit operations
$u,v\in\Om XU$ on a suitable ambient pseudovariety \pv U over a fixed
alphabet~$X$, whose letters are called the \emph{variables} of the
system. We may say that $\Sigma$ consists of \emph{\pv U-equations}%
\index{$\mathsf{U}$-equations} %
\index{equations!$\mathsf{U}$--} %
to emphasize this condition. Additionally, we impose for each variable
$x$ a clopen constraint $K_x\subseteq\Om AU$ over another fixed
alphabet. The constraints are thus recognizable subsets of~\Om AU. We
say that the constrained system has a \emph{solution}%
\index{equations!solution of system of --} %
$\gamma$ in an $A$-generated pro-\pv U algebra $T$ if $\gamma:X\to\Om
AU$ is a function such that the following two conditions hold, where
$\hat\gamma:\Om XU\to\Om AU$ and $\pi:\Om AU\to T$ are the unique
continuous homomorphisms respectively extending $\gamma$ and respecting
the choice of generators of~$T$:
\begin{enumerate}
\item for each variable $x\in X$, the constraint $\gamma(x)\in K_x$ is
  satisfied;
\item for each equation $u=v$ in~$\Sigma$, the equality
  $\pi(\hat\gamma(u))=\pi(\hat\gamma(v))$ holds.
\end{enumerate}
The following is a simple compactness result which can be found for
instance in~\cite{Almeida:2002a}.

\begin{theorem}\label{t:ALCO-systems-compactness-solutions}
  A system of\/ \pv U-equations over a set of variables $X$ with
  clopen constraints $K_x\subseteq\Om AU$ ($x\in X$) has a solution in
  every $A$-generated algebra from a given subpseudovariety \pv V
  of\/~\pv U if and only if it has a solution in~\Om AV.
\end{theorem}

If the set of variables $X$~is finite, which we assume from hereon,
then there is a continuous homomorphism $\varphi:\Om AU\to S$ into a
finite algebra $S$ which recognizes all the given constraints
$K_x\subseteq\Om AU$ ($x\in X$). Then the existence of a solution for
the system in an $A$-generated algebra $T\in\pv U$ is equivalent to
the existence of a solution in~$T$ for the same system for at least
one of a certain set of constraints of the form $K'_x=\varphi^{-1}(s)$
with $s\in S$. Thus, one may prefer to give the constraints in the form
of a function $X\to S$ into an $A$-generated finite algebra~$S$.

Another formulation of the above ideas is in terms of relational
morphisms, which is the perspective initially taken
in~\cite{Almeida:1999b} and which prevails
in~\cite{Rhodes&Steinberg:2009qt}. A \emph{relational morphism}%
\index{relational morphism}%
\index{morphism!relational --} %
between two topological algebras $S$ and $T$ is a closed subalgebra
$\mu$ of the direct product $S\times T$ whose projection in the first
component is onto. Note that, if $S$ and $T$ are pro-\pv U algebras
then so is $\mu$ and if $\mu$ is $A$-generated, then the induced
continuous homomorphisms $\varphi:\Om AU\to S$ and $\psi:\Om AU\to T$
are such that $\mu$ is obtained by composing the relations
$\varphi^{-1}\subseteq S\times\Om AU$ and $\psi\subseteq\Om AU\times
T$. This is called a \emph{canonical factorization} of~$\mu$.

An example of such a relational morphism is obtained as follows. Let
$\varphi:A\to S$ be a generating mapping for a pro-\pv U algebra $S$ and
let \pv V be a subpseudovariety of~\pv U. Consider the unique
continuous homomorphisms $\hat\varphi:\Om AU\to S$ and $\psi:\Om AU\to\Om
AV$ respecting the choice of generators. Then $\mu_{\pv
  V,A}=\hat\varphi^{-1}\psi$ is a relational morphism from $S$ to~\Om AV.

We say that the system of \pv U-equations $\Sigma$ with constraints
given by a function $\xi:X\to S$ into a finite algebra $S$ is
\emph{inevitable} with respect to a relational morphism $\mu\subseteq
S\times T$, where $T$~is a profinite algebra, if there is a continuous
homomorphism $\delta:\Om XU\to T$ such that the following conditions
hold:
\begin{enumerate}
\item for each variable $x\in X$, the constraint
  $(\xi(x),\delta(x))\in\mu$ is satisfied;
\item for each equation $u=v$ in~$\Sigma$, the equality
  $\delta(u)=\delta(v)$ holds.
\end{enumerate}
One can easily check that this property is equivalent to the existence
of a solution of the system subject to the constraints
$K_x=\hat\varphi^{-1}(\xi(x))\subseteq\Om AU$, where
$\mu=\hat\varphi^{-1}\psi$ is the canonical factorization associated with
a finite generating set~$A$ for~$\mu$.
Theorem~\ref{t:ALCO-systems-compactness-solutions} then yields the
following similar compactness theorem for inevitability.

\begin{theorem}\label{t:ALCO-systems-compactness-inevitability}
  For a system of\/ \pv U-equations over a finite set $X$, of
  variables, with constraints given by a
  mapping $X\to S$ into a finite algebra $S$ and a subpseudovariety
  \pv V of\/~\pv U, the following conditions are equivalent:
  \begin{enumerate}
  \item the constrained system is inevitable with respect to every
    relational morphism $\mu$ from $S$ into an arbitrary algebra
    from~\pv V;
  \item the constrained system is inevitable with respect to every
    relational morphism $\mu$ from $S$ into an arbitrary pro-\pv V
    algebra;
  \item for \emph{some} finite generating set $A$ of~$S$, the constrained
    system is inevitable with respect to the relational morphism
    $\mu_{\pv V,A}$;
  \item for \emph{every} finite generating set $A$ of~$S$, the constrained
    system is inevitable with respect to the relational morphism
    $\mu_{\pv V,A}$.
  \end{enumerate}
\end{theorem}

Let \pv V be a subpseudovariety of~\pv U. We say that a constrained
system is \emph{\pv V-inevitable} if it satisfies the equivalent
conditions of Theorem~\ref{t:ALCO-systems-compactness-inevitability}.
The pseudovariety \pv V is said to be \emph{hyperdecidable} with
respect to a class \Cl S of systems of \pv U-equations with
constraints in algebras from~\pv U if there is an algorithm that
decides, for each constrained system in~\Cl S, whether it is \pv
V-inevitable.

An approach to prove hyperdecidability which was devised by Steinberg
and the first author
\cite{Almeida&Steinberg:2000a,Almeida&Steinberg:2000b}, inspired by
seminal work of Ash~\cite{Ash:1991}, was to draw this property from
other either more familiar or more conceptual properties. Assume that
the class \Cl S consists of finite systems, that it is recursively
enumerable, and that the implicit operations that appear on the sides
of the equations of the systems are computable. Moreover, suppose that
\pv V is recursively enumerable. One can then effectively check
whether a constrained system in~\Cl S is inevitable with respect to a
relational morphism from the constraining algebra into an algebra
from~\pv V, which gives a semi-algorithm to enumerate the constrained
systems which are not \pv V-inevitable. To decide whether a constrained
system from~\Cl S has a solution in~\Om AV it thus suffices to add
hypotheses to guarantee that there is also a semi-algorithm to
enumerate the systems that are \pv V-inevitable. To do so, the idea is
to prove that if the system is \pv V-inevitable, then there is a
solution of a special kind, so that the candidates for such special
solutions can be effectively enumerated and whether such a candidate
is indeed a solution can be effectively checked.

To formalize this idea, consider a recursively enumerable set $\tau$
of computable implicit operations on~\pv U, including the basic
operations. We call such a set $\tau$ a \emph{computable implicit
  signature}%
\index{computable implicit signature}%
\index{implicit signature!computable --}%
\index{signature!computable implicit --} %
over~\pv U. Note that every pro-\pv U algebra has automatically the
structure of a $\tau$-algebra (see
Proposition~\ref{p:ALCO-implicit-operations}). For a subpseudovariety
\pv V of~\pv U, we denote by \omt AV%
\index{$\Omega_A^\tau\mathsf{V}$} %
the $\tau$-subalgebra of~\Om AV generated by~$A$. It follows from the
definition of free pro-\pv V algebra that \omt AV is freely generated
by~$A$ in the variety of $\tau$-algebras generated by~\pv V. The
\emph{word problem}%
\index{word problem} %
for \omt AV consists in, given two $\tau$-terms over the alphabet $A$,
deciding whether they represent the same element of~\omt AV. We may
now state the following key definition.

\begin{definition}\label{d:ALCO-tameness}
  Let \pv V be a recursively enumerable subpseudovariety of~\pv U and
  let \Cl S be a class of constrained systems of \pv U-equations. We
  say that \pv V is \emph{$\tau$-reducible}%
  \index{reducible}%
  \index{equations!reducible system of --} %
  with respect to~\Cl S if, whenever a constrained system in~\Cl S has
  a solution $\gamma:X\to\Om AU$ in~\Om AV, it has a solution
  $\gamma':X\to\omt AU$ in~\Om AV.\footnote{A topological formulation
    of the notion of $\tau$-reducibility was recently found in
    \cite{Almeida&Klima:2017a}. It simply states that, for each system from
    \Cl S, forgetting the constraints, the solutions in \Om AV from
    \Cl S taking values in \omt AU are dense in the set of all
    solutions in~\Om AV.} If, moreover, the word problem for~\omt AV
  is decidable, then we say that \pv V is \emph{$\tau$-tame}%
  \index{tame}%
  \index{pseudovariety!tame --} %
  with respect to~\Cl S. We say that \pv V is \emph{completely
    $\tau$-tame}%
  \index{completely $\tau$-tame}%
  \index{pseudovariety!completely $\tau$-tame --} %
  if it is $\tau$-tame with respect to the class of all finite
  constrained systems of equations of $\tau$-terms.
\end{definition}

The following result summarizes the above discussion.

\begin{theorem}\label{t:ALCO-tame-vs}
  Let \pv U be a recursively enumerable pseudovariety and let $\tau$
  be a computable implicit signature over~\pv U. Let \Cl S be a
  recursively enumerable class of constrained systems of equations
  between $\tau$-terms. Finally, let \pv V be a subpseudovariety
  of\/~\pv U. If\/ \pv V~is $\tau$-tame with respect to~\Cl S, then
  \pv V~is hyperdecidable with respect to~\Cl S.
\end{theorem}

Several important examples of tame pseudovarieties are discussed in
Subsection~\ref{sec:ALCO-tameness-pvs}. Here, we only present tameness
results which hold in the general algebraic context to which this
section is dedicated. Before doing so, we introduce a weaker version
of tameness which is also of interest.

Let $S$ be an $A$-generated algebra from~\pv U and let $\tau$ be a
computable implicit signature. The relational morphism $\bar\mu_{\pv
  V,A}^\tau\subseteq S\times\omt AV$ is obtained by taking the
intersection of~$\mu_{\pv V,A}$ with $S\times\omt AV$. We say that \pv
V is \emph{weakly $\tau$-reducible}%
\index{weakly reducible pseudovariety}%
\index{pseudovariety!weakly reducible --} %
for a class \Cl S of constrained systems of \pv U-equations if, for
every \pv V-inevitable constrained system in~\Cl S, say with
constraints in the $A$-generated algebra $S\in\pv U$, the system is
inevitable with respect to the relational morphism $\bar\mu_{\pv
  V,A}^\tau$. Replacing $\tau$-reducibility by weak
$\tau$-reducibility in the definition of $\tau$-tameness we speak of
\emph{weak $\tau$-tameness}%
\index{weak tameness}%
\index{tame!weakly --}%
\index{pseudovariety!weakly tame --}. %

Viewing \omt AV as a discrete algebra, there is another natural
relational morphism $\mu_{\pv V,A}^\tau%
\index{$\mu_{\pv V,A}^\tau$} %
\subseteq S\times\omt AV$, namely the $\tau$-subalgebra generated by
the pairs of the form $(a,a)$ with $a\in A$. The notation is justified
since, as it is easily proved, the relation $\bar\mu_{\pv V,A}^\tau$%
\index{$\bar\mu_{\pv V,A}^\tau$} %
is the closure of $\mu_{\pv V,A}^\tau$ in $S\times\omt AV$ with
respect to the discrete topology in the first component and the
pro-\pv V topology in the second component. We say that \pv V is
\emph{$\tau$-full}%
\index{pseudovariety!$\tau$-full --} %
if the two relational morphisms coincide for every $A$-generated
algebra $S$ from~\pv U. Note that a weakly $\tau$-reducible
$\tau$-full pseudovariety is $\tau$-reducible. Conversely, the
terminology is justified by the fact that, if \pv V is
$\tau$-reducible with respect to a constrained system $\Sigma$ of \pv
U-equations, then it is also weakly $\tau$-reducible with respect
to~$\Sigma$.

We say that the pseudovariety \pv V has \emph{computable
  $\tau$-closures}%
\index{pseudovariety!has computable $\tau$-closures} %
if there is an algorithm such that, given a finite alphabet $A$, a
regular subset $L$ of~\om AV and an element $v\in\omt AV$, determines
whether or not $v$ belongs to the closure of $L$ in the pro-\pv V
topology of~\omt AV. The following combines a couple of results from
\cite{Almeida&Steinberg:2000a}.

\begin{theorem}\label{t:ALCO-weak-reducibility}
  Let \pv V be a recursively enumerable subpseudovariety of a
  recursively enumerable pseudovariety \/~\pv U, let
  $\tau$ be a computable implicit signature, and suppose that the word
  problem for each \omt AV is decidable.
  \begin{enumerate}
  \item If\/ \pv V is $\tau$-full then \pv V has computable $\tau$-closures.
  \item If\/ \pv V is weakly $\tau$-reducible for a class \Cl S of
    constrained systems of\/ \pv U-equations and \pv V has computable
    $\tau$-closures, then \pv V is hyperdecidable with respect to~\Cl
    S.
  \end{enumerate}
\end{theorem}

We say that a class of algebras is \emph{locally finite}%
\index{locally finite} %
\index{algebras!locally finite class of --} %
if all finitely generated algebras in the variety it generates are
finite. This is the case, for instance, for a pseudovariety generated
by a single algebra but not every locally finite pseudovariety is of
this kind. A well-known example in the realm of semigroups is provided
by the pseudovariety of all finite bands (in which every element is
idempotent).

A decidable locally finite pseudovariety \pv V is said to be
\emph{order computable}%
\index{pseudovariety!order computable --} %
if the function that associates with each positive integer $n$ the
cardinality of the algebra \Om nV is computable. It seems to be an
open problem whether every locally finite pseudovariety is order
computable. The following result is an immediate extension
of~\cite[Theorem~4.18]{Almeida&Steinberg:2000a}, which is based on the
``slice theorem'' of Steinberg~\cite{Steinberg:2001b}.

\begin{theorem}\label{t:ALCO-tameness_vs_join_with_order-computable}
  Let \pv V be a $\tau$-tame pseudovariety with respect to a class
  \Cl S of systems of equations and let \pv W be an order-computable
  pseudovariety. Then the join $\pv V\vee\pv W$ is also $\tau$-tame
  with respect to~\Cl S.
\end{theorem}

One of the ingredients behind the proof of
Theorem~\ref{t:ALCO-tameness_vs_join_with_order-computable} is that
$\Om AW=\omt AW$ for every locally finite pseudovariety \pv W and
every implicit signature $\tau$. Under this weaker property for a
computable implicit signature $\tau$, tameness becomes much simpler.
The following result is a simple corollary of some of the above
results. We do not know whether the $\tau$-fullness hypothesis can be
dropped.

\begin{proposition}\label{p:ALCO-tameness_under_automatic_reducibility}
  Let $\tau$ be a computable implicit signature and let \pv V be a
  recursively enumerable pseudovariety such that the equality $\Om
  AV=\omt AV$ holds for every finite set~$A$ and \pv V is $\tau$-full.
  Then \pv V is completely $\tau$-tame if and only if the word problem
  for each \omt AV is decidable.
\end{proposition}

\section{The case of semigroups}
\label{sec:ALCO-sgps}

The motivation to study profinite topologies in finite semigroup
theory comes from automata and language theory: Eilenberg's
correspondence theorem \cite{Eilenberg:1976} shows the relevance of
investigating pseudovarieties of semigroups and monoids.

The results mentioned in this section by no means cover entirely 
the literature in the area that is presently available. In
particular, we stick to the more classical case of semigroups, while
for instance the cases of ordered semigroups or stamps have come to
play a significant role, as can be seen in Chapter~16. It
turns out that in all these cases the same relatively free profinite
semigroups intervene and the results are also often quite similar,
although sometimes their proofs involve additional technical
difficulties.

\subsection{Computing profinite closures}
\label{sec:ALCO-prof-closures}

There are several reasons why profinite topologies are relevant for
automata theory and Sections~\ref{sec:ALCO-intro}
and~\ref{sec:ALCO-profinite-top-algebras} provide many of them.
We start this subsection by formulating a simple problem which has a
direct translation in terms of profinite topologies.

Let $A$ be a finite alphabet and let $L\subseteq A^+$ be a regular
language. Membership in~$L$ of a word $w\in A^+$ can be effectively
tested by checking whether the action of $w$ on the initial state of
the minimal automaton of~$L$ leads to a final state, or whether the
syntactic image of $w$ belongs to that of~$L$. But, one may be
interested in a weaker test such as whether $w$ may be separated
from~$L$ by a regular language $K$ of a particular type, for instance
a group language (i.e., a language whose syntactic semigroup is a
group): does there exist a group language $K\subseteq A^+$ such that
$w\in K$ and $K\cap L=\emptyset$?

Let \pv G%
\index{$\mathsf{G}_p$} %
denote the pseudovariety of all finite groups. In view of
Proposition~\ref{p:ALCO-pro-Q-uniformity}\eqref{item:ALCO-pro-Q-uniformity-3},
the above separation property is equivalent to being able to separate
$w$ from~$L$ by some open set in the pro-\pv G topology of~$A^+$.
Thus, in terms of the pro-\pv G topology, the above question
translates into testing membership of $w$ in the closure of~$L$ in the
pro-\pv G topology of~$A^+$. More generally, we have the following
result, where we denote by \clos VL%
\index{$\mathrm{cl}_{\mathsf{V}(L)}$} %
the closure of $L$ in the pro-\pv V topology of~$A^+$.

\begin{proposition}\label{p:ALCO-separation}
  Let $A$ be a finite alphabet and let \pv V be a pseudovariety of
  semigroups. For a regular language $L\subseteq A^+$, a word $w\in
  A^+$ can be separated from~$L$ by a \pv V-recognizable language if
  and only if $w\notin\clos VL$.
\end{proposition}

Thus, for a pseudovariety \pv V to have computable $\tau$-closures for
the implicit signature reduced to multiplication (see
Subsection~\ref{sec:ALCO-decid-tamen}) is a property that has
immediate automata-theoretic relevance.

The special case of separation by group languages has particular
historical importance. It was first considered by Pin and Reutenauer
\cite{Pin&Reutenauer:1991}, who proposed the following recursive
procedure to compute \clos GL, where $FG(A)$ denotes the group freely
generated by~$A$.

\begin{theorem}[{\cite[Theorem~2.4]{Pin&Reutenauer:1991}}]
  \label{t:ALCO-Pin-Reutenauer}
  Given a regular expression for a language~$L\subseteq A^+$, replace
  the operation $K\mapsto K^+$ by that of taking the subgroup of
  $FG(A)$ generated by the argument $K$. The resulting expression
  describes a subset of~$FG(A)$ and \clos GL is its intersection
  with~$A^+$.
\end{theorem}

The correctness of the algorithm described in
Theorem~\ref{t:ALCO-Pin-Reutenauer} was reduced to the proposed
conjecture that the product of finitely many finitely generated
subgroups of a free group is closed in its profinite topology, thus
generalizing M. Hall's result that finitely generated subgroups of the
free group are closed in the profinite topology \cite{Hall:1950}. This
conjecture was established by Ribes and
Zalesski\u\i~\cite{Ribes&Zalesskii:1993a} using profinite group
theory. The original motivation for computing \clos GL comes from the
fact that Pin and Reutenauer also showed that the correctness of their
procedure implies that the ``type II conjecture''%
\index{conjecture!type II --}%
\index{type II conjecture} %
holds. This other conjecture gives a constructive description of
the \emph{group kernel}%
\index{group kernel} %
\index{monoid!group kernel of --} %
$K_\pv G(M)$ of a finite monoid $M$. More generally, for a
pseudovariety \pv H of groups, the \emph{\pv H-kernel} $K_\pv H(M)$%
\index{$K_\pv H(M)$} %
consists of all $m\in M$ such that, for every relational morphism
$\mu$ from $M$ to a group in~\pv H, $(m,1)$ belongs to~$\mu$. The type
II conjecture states that $K_\pv G(M)$ is the smallest submonoid
of~$M$ that contains the idempotents and is closed under the
operation that sends $m$ to $amb$ if $aba=a$ or $bab=b$ (\emph{weak
  conjugation}%
\index{weak conjugation}%
\index{conjugation!weak --}). %
An independent proof of the type II conjecture was
obtained by Ash~\cite{Ash:1991} and is discussed in
Subsection~\ref{sec:ALCO-tameness-pvs}.

The pro-\pv V closure of regular languages in $A^+$ has also been
considered for other pseudovarieties \pv V. For a pseudovariety \pv H
of groups, the motivation comes from the membership problem in the
pseudovariety $\pv W\malcev\pv H$ for a pseudovariety of monoids \pv
W. Indeed, it is easy to show that a finite monoid $M$ belongs to this
Mal'cev product if and only if $K_\pv H(M)$ belongs to~\pv W. On the
other hand, Pin~\cite{Pin:1991} observed that, if $\varphi:A^*\to M$
is an onto homomorphism and $m\in M$, then $m\in K_\pv H(M)$ if and
only if the empty word $1$ belongs to $\clos
H{\varphi^{-1}(m)}$. Thus, we have the following result.

\begin{proposition}\label{p:ALCO-deciding-W_malcev_H}
  Let \pv W be a decidable pseudovariety of monoids and \pv H a
  pseudovariety of groups such that one can decide whether, given a
  regular language $L\subseteq A^*$, the empty word belongs to~\clos
  HL. Then $\pv W\malcev\pv H$ is decidable.
\end{proposition}

The problem of computing the pro-\pv H closure in~$A^*$ of a regular
language $L\subseteq A^*$ has been considered for other
pseudovarieties of groups such as \pv{Ab} (Abelian groups)
\cite{Delgado:2001}, $\pv G_p$%
\index{$\mathsf{G}_p$} %
($p$-groups), $\pv G_{\mathrm{nil}}$%
\index{$\mathsf{G}_{\mathrm{nil}}$} %
(nilpotent groups), and $\pv G_{\mathrm{sol}}$%
\index{$\mathsf{G}_{\mathrm{sol}}$} %
(solvable groups) \cite{Ribes&Zalesskii:1993b,Margolis&Sapir&Weil:1999}.

Suppose that \pv H is a pseudovariety of groups such that the free
group $FG(A)$ is residually in~\pv H. Then we have natural embeddings
$A^*\hookrightarrow FG(A)\hookrightarrow \Om AH$. The pro-\pv H
topology of a subalgebra of~\Om AH is its subspace topology by
Propositions~\ref{p:ALCO-pro-Q-completion-is-pro-Q}
and~\ref{p:ALCO-free-pro-Q-through-completion}. Thus, an equivalent
problem to computing the \pv H-kernel of a finite monoid is to decide
whether $1$~belongs to the closure in~$FG(A)$ of a regular language
$L\subseteq A^*$. In case \pv H is closed under extensions (or,
equivalently, under semidirect product), Ribes and
Zalesski\u\i~\cite[Theorem~5.1]{Ribes&Zalesskii:1993b} have shown
that, for the pro-\pv H topology, the product of finitely many
finitely generated closed subgroups of~$FG(A)$ is also closed. Using
the Pin-Reutenauer techniques, they reduced the computation of the
pro-\pv H closure of a regular language $L\subseteq A^*$ to the
computation of the pro-\pv H closure in~$FG(A)$ of a given finitely
generated subgroup. That these results also apply
to $\pv G_{\mathrm{nil}}$ has been recently
shown in~\cite{Almeida&Shahzamanian&Steinberg:2017}.
Algorithms for the computation of
the pro-$\pv G_p$ and pro-$\pv G_{\mathrm{nil}}$ closures of finitely
generated subgroups of a free group can be found in
\cite{Ribes&Zalesskii:1993b,Margolis&Sapir&Weil:1999}. The case of
$\pv G_{\mathrm{sol}}$ remains open.

For an element $s$ of a profinite semigroup $S$, $s^\omega$%
\index{$\_^\omega$} %
denotes the unique idempotent in the closed subsemigroup $T$ generated
by~$s$, and $s^{\omega-1}$%
\index{$\_^{\omega-1}$} %
the inverse of $s^{\omega+1}=ss^\omega$ in the maximal subgroup
of~$T$. Consider the implicit signature $\kappa$%
\index{$\kappa$}%
\index{signature!$\kappa$}%
\index{implicit signature!$\kappa$} %
consisting of multiplication together with the unary operation
$x\mapsto x^{\omega-1}$. The free group $FG(A)$ may be then identified
with the free algebra \omc AG. This suggests generalizations of the
Pin-Reutenauer procedure for computing the pro-\pv G closure of a
regular language $L\subseteq A^*$ to other pseudovarieties. The analog
of the procedure is shown in~\cite{Almeida&Costa&Zeitoun:2014} to hold
for the pseudovariety \pv A (and the signature~$\kappa$).

Another important application of the separation problem has been given
by Place and Zeitoun \cite{Place&Zeitoun:2014b} in the study of the
decidability problem for the Straubing-Th\'erien hierarchy of
star-free languages.

\subsection{Tameness}
\label{sec:ALCO-tameness-pvs}

There is a natural way of associating a system of semigroup equations
to a finite digraph that is relevant for the computation of
semidirect products of pseudovarieties of semigroups
\cite{Almeida:1999b}. Namely, the variables of the system are the
vertices and the arrows, and each arrow $u\xrightarrow{e}v$ gives rise
to an equation $ue=v$ which relates the act of following the arrow
with multiplication as in a Cayley graph. The term \emph{tameness} was
introduced in~\cite{Almeida&Steinberg:2000b} to refer to tameness with
respect to such systems of equations in the sense of
Subsection~\ref{sec:ALCO-decid-tamen}. To avoid confusion, we prefer
to call it here \emph{graph tameness}%
\index{tame!graph --}%
\index{pseudovariety!graph tame --}. %
We adopt a similar convention for other properties parametrized by
systems of equations, such as hyperdecidability and reducibility.

For example, for the one-vertex one-loop digraph, the corresponding
equation is $xy=x$. It is easy to verify that, with constraints given
by a function $\xi$ into a finite monoid $M$, this equation is \pv
G-inevitable if and only if $\xi(y)\in K_\pv G(M)$. For the two-vertex
digraph with $n$ arrows from one vertex to the other one, the
associated system of equations has the form $xy_i=z$ ($i=1,\ldots,n$).
If $\xi$ is a constraining function into a finite monoid $M$ such that
$\xi(x)=1$, then the system is \pv G-inevitable if and only if, for
every relational morphism $\mu$ from $M$ into an arbitrary $G\in \pv
G$, there is some $g\in G$ such that
$\xi(\{y_1,\ldots,y_n,z\})\times\{g\}\subseteq\mu$. Replacing \pv G by
an arbitrary pseudovariety \pv V of monoids, the latter condition is
expressed by saying that the subset $\xi(\{y_1,\ldots,y_n,z\})$ of~$M$
is \emph{\pv V-pointlike}%
\index{subset!pointlike --}. %

The first and best known example of a graph tame pseudovariety is that
of the pseudovariety \pv G. This result has been discovered in
different disguises, first by Ash~\cite{Ash:1991}, as a means of
establishing the type II%
\index{conjecture!type II --}%
\index{type II conjecture} %
(see Subsection~\ref{sec:ALCO-prof-closures}
and~\cite{Karnofsky&Rhodes:1982}) and \emph{pointlike}%
\index{conjecture!pointlike --}%
\index{pointlike conjecture} %
\cite{Henckell&Rhodes:1991} conjectures. In Ash's formulation, the
arrows of finite digraphs are labeled with elements of a finite monoid
$M$ and the result is said to be \emph{inevitable} if, for every
relational morphism $\mu$ from $M$ into an arbitrary finite group $G$,
each label can be replaced by a $\mu$-related label in~$G$ such that,
for every (not necessarily directed) cycle, the product of the labels
of the arrows, or their inverses for backward arrows, is equal to~1
in~$G$. In the notation of Subsection~\ref{sec:ALCO-decid-tamen} and
also taking into account \cite[Lemma~4.8]{Almeida&Steinberg:2000a},
Ash's theorem states that such a labeled digraph is inevitable if and
only if the preceding property holds for the relational morphism
$\mu_{\pv G,A}^\kappa$ associated with any choice of generating set
$A$ for~$M$. It then follows easily that Ash's theorem translates to
the statement that \pv G is graph $\kappa$-tame.

Fix a finite relational language. A class \Cl R of relational
structures is said to satisfy the \emph{finite extension property for
  partial automorphisms} (FEPPA) if, for every finite structure $R$
in~\Cl R and every set $P$ of isomorphisms between substructures
of~$R$, if there exists in~\Cl R an extension $S$ of~$R$ in which all
$f\in P$ extend to automorphisms of~$S$, then there is such an
extension $S$ which is finite. A \emph{homomorphism} of relational
structures is a function that preserves the relations in the forward
direction. The \emph{exclusion} of a class \Cl R of relational
structures is the class of relational structures $S$ such that there
is no homomorphism $R\to S$ with $R\in\Cl R$. Herwig and Lascar
\cite{Herwig&Lascar:1997} showed that, for a finite class \Cl R of
finite relational structures, its exclusion class satisfies FEPPA. They also
gave an equivalent formulation of this result in terms of a property
of free groups, which Delgado and the first author
\cite{Almeida&Delgado:1997,Almeida&Delgado:1999} proved to be
equivalent to the graph $\kappa$-tameness of~\pv G.

On the other hand, it follows from results of Coulbois and Khélif
\cite{Coulbois&Khelif:1999} that the pseudovariety \pv G is not
completely $\kappa$-tame. It would be of interest to find a signature
$\tau$ such that \pv G is completely $\tau$-tame, if any such
signature exists.

\begin{theorem}
  \label{t:ALCO-G-tame}
  The pseudovariety \pv G is graph $\kappa$-tame but not completely
  $\kappa$-tame.
\end{theorem}

Tameness has also been investigated for other pseudovarieties of
groups. The pseudovariety \pv{Ab} is completely $\kappa$-tame
\cite{Almeida&Delgado:2001}. On the other hand, a pseudovariety of
Abelian groups is completely hyperdecidable if and only if it is
decidable while it is completely $\kappa$-tame if and only if it is
locally finite or \pv{Ab} \cite{Delgado&Masuda&Steinberg:2007}. For
the pseudovariety $\pv G_p$, the situation is more complicated.
Steinberg \cite[Theorem~11.12]{Steinberg:1998a} showed that, for every
nontrivial extension-closed pseudovariety of groups \pv H such that
the pro-\pv H closure of a finitely generated subgroup of a free group
is again finitely generated, \pv H is graph weakly $\kappa$-reducible.
On the other hand, a graph $\kappa$-reducible pseudovariety must admit
a basis of pseudoidentities consisting of $\kappa$-identities
\cite[Proposition~4.2]{Almeida&Steinberg:2000a} which, since free
groups are residually in~$\pv G_p$, entails that $\pv G_p$ is not
graph $\kappa$-tame. Using symbolic dynamics techniques to generate a
suitable infinite implicit signature, the first author has established
the following result \cite{Almeida:1999c}.

\begin{theorem}
  \label{t:ALCO-Gp-tame}
  There is a signature $\tau$ such that $\pv G_p$ is graph
  $\tau$-tame.
\end{theorem}

Building on the approach of \cite{Almeida:1999c}, Alibabaei has
constructed for each decidable pseudovariety \pv H of Abelian groups
an implicit signature with respect to which \pv H is completely tame 
\cite{Alibabaei:2017b} and also an implicit signature with respect to
which $\pv G_{\mathrm{nil}}$ is graph tame \cite{Alibabaei:2017c}.

A semigroup is said to be \emph{completely regular}%
\index{semigroup!completely regular --} %
if every element lies in some subgroup. The pseudovariety \pv{CR}%
\index{$\mathsf{CR}$} %
consists of all completely regular finite monoids and \pv{OCR}%
\index{$\mathsf{OCR}$} %
is the subpseudovariety consisting of those in which the idempotents
constitute a submonoid. Both these pseudovarieties have been shown to
be graph $\kappa$-tame
\cite{Almeida&Trotter:2001,Almeida&Trotter:1999a},\footnote{The
  conjecture to which the graph tameness of \pv{CR} is reduced
  in~\cite{Almeida&Trotter:2001} has been observed by K. Auinger
  (private communication) to hold using the methods of
  \cite{Almeida&Delgado:1997,Almeida&Delgado:1999}. There is also
  another difficulty which comes from the fact that free profinite
  semigroupoids over profinite graphs are considered. As has been
  shown in~\cite{Almeida&ACosta:2007a}, there are some rather delicate
  aspects in the description of such structures when the graph has
  infinitely many vertices, namely the free subsemigroupoid generated
  by a dense subgraph of the profinite graph may not be dense, and one
  needs in general to transfinitely iterate algebraic and topological
  closures. However, one can check that for the free profinite
  semigroupoid in question, the iteration stops in one step, from
  which it follows that the required properties of the suitable free
  profinite semigroupoid are guaranteed \cite{Almeida&Costa:2015a}.}
results which depend heavily on Theorem~\ref{t:ALCO-G-tame} together
with structure theorems for the corresponding relatively free
profinite monoids.

Several aperiodic pseudovarieties have also been investigated. An
interesting example is that of the pseudovariety \pv J%
\index{$\mathsf{J}$} %
of all finite \Cl J-trivial semigroups, corresponding to the variety
of piecewise testable languages%
\index{language!piecewise testable --} %
(see~\cite{Eilenberg:1976}). The first author has
shown that, for a finite alphabet $A$, $\Om AJ=\omc AJ$ and also
solved the word problem for \omc AJ (see
\cite[Section~8.2]{Almeida:1994a}). Since it is an easy exercise to
deduce that \pv J is $\kappa$-full, it follows from
Proposition~\ref{p:ALCO-tameness_under_automatic_reducibility} that
\pv J is completely $\kappa$-tame, and therefore graph hyperdecidable.
The construction of a ``real algorithm'' to decide inevitability turns
out to be much more involved \cite{Almeida&Zeitoun:1998}.

For the pseudovariety \pv R, consisting of all finite \Cl R-trivial
semigroups, constructing a concrete algorithm to show that \pv R is
graph hyperdecidable is technically complicated, even when only
strongly connected digraphs are considered \cite{Almeida&Silva:1997b}.
Building on seminal ideas of Makanin \cite{Makanin:1977} and taking
into account the structure of free pro-\pv R semigroups
\cite{Almeida&Weil:1996b}, Costa, Zeitoun and the first
author~\cite{Almeida&Costa&Zeitoun:2005b} have established the
following result.

\begin{theorem}
  \label{t:ALCO-R-tame}
  The pseudovariety \pv R is completely $\kappa$-tame.
\end{theorem}

This result has been extended in \cite{Almeida&Borlido:2017} to
pseudovarieties of the form DRH, consisting of all finite semigroups
in which every regular \Cl R-class is a group from the pseudovariety
\pv H of groups.

Consider next the pseudovariety \pv{LSl}%
\index{$\mathsf{LSl}$} %
of all finite local semilattices, which corresponds to the variety of
locally testable languages (see~\cite{Eilenberg:1976}). The proof of
the following result involves very delicate combinatorics on words
\cite{Costa&Teixeira:2004,Costa&Nogueira:2009}.

\begin{theorem}
  \label{t:ALCO-LSl-tame}
  The pseudovariety \pv{LSl} is completely $\kappa$-tame.
\end{theorem}

J. Rhodes announced in a conference held in 1998 in Lincoln, Nebraska,
that \pv A is graph $\kappa$-tame. The only part of the program to
establish such a result that has been published is McCammond's
solution of the word problem for \omc AA \cite{McCammond:1999a}.
Another, earlier ingredient in Rhodes' ideas comes from Henckell's
computation of the \pv A-pointlike subsets of a given finite semigroup
\cite{Henckell:1988}. See \cite{Place&Zeitoun:2014a} for a different proof.
In~\cite{Henckell&Rhodes&Steinberg:2010} there is also an alternative
proof and the following generalization.

\begin{theorem}
  \label{t:ALCO-A-pointlikes}
  If $\pi$ is a recursive set of prime integers, then there is an
  algorithm to compute pointlike sets of finite semigroups with
  respect to the pseudovariety $\overline{\pv G_\pi}$, consisting of
  all finite semigroups whose subgroups are $\pi$-groups.
\end{theorem}

Once it was discovered that there was a gap in the proof of the
\emph{basis theorem}%
\index{basis theorem} %
(see the discussion in Subsection~\ref{sec:ALCO-decid-tamen}), which
invalidated the reduction of the decidability of the Krohn-Rhodes
complexity to proving that \pv A is tame announced
in~\cite{Almeida&Steinberg:2000a}, Rhodes withdrew several manuscripts
that he claimed would prove that \pv A is tame. The Krohn-Rhodes complexity%
\index{Krohn-Rhodes complexity} %
pseudovarieties are defined recursively by $\pv C_0=\pv A$ and $\pv
C_{n+1}=\pv C_n*\pv G*\pv A$ \cite{Krohn&Rhodes:1968}, which
determines a complete and strict hierarchy for the pseudovariety of
all finite semigroups. Here, $*$ denotes the \emph{semidirect
  product}%
\index{semidirect product}%
\index{semigroups!semidirect product of pseudovarieties of --} %
of pseudovarieties of semigroups as defined
in~\cite[Section~10.1]{Almeida:1994a}, which is an associative
operation.

It has also been investigated whether tameness is preserved under the
operators of join and semidirect product. Since tameness is apparently
much stronger than decidability, if tameness is preserved by
semidirect product then the decidability of the Krohn-Rhodes
complexity is indeed reduced to proving that \pv A is tame. But, so
far, only very special cases have been treated. An example is the
following result \cite{Almeida&Costa&Teixeira:2010}, which improves
on~\cite{Almeida&Silva:1997a}.

\begin{theorem}
  \label{t:ALCO-*_vs_tameness}
  Let \pv V be a graph $\kappa$-tame pseudovariety and let \pv W be an
  order computable pseudovariety. Then $\pv V*\pv W$ is graph
  $\kappa$-tame.
\end{theorem}

It is also unknown whether tameness is preserved under join. Yet,
several positive results have been obtained. The following theorem
combines results from
\cite{Almeida&Steinberg:2000a,Almeida&Costa&Zeitoun:2004}.

\begin{theorem}
  \label{t:ALCO-v_vs_tameness}
  Let \Cl C be a class of constrained systems of equations.
  \begin{enumerate}
  \item Let \pv V be an order-computable pseudovariety. If a
    pseudovariety \pv W is $\tau$-tame with respect to~\Cl C, then so
    is $\pv V\vee\pv W$.
  \item Let \pv V be a recursively enumerable $\kappa$-full
    subpseudovariety of~\pv J such that the word problem for \omc AV
    is decidable. If a pseudovariety \pv W is $\tau$-tame with respect
    to~\Cl C, then so is $\pv V\vee\pv W$.
  \item Let \pv W be a pseudovariety satisfying some pseudoidentity
    of the form
    $$x_1\cdots x_n y^{\omega+1}zt^\omega=x_1\cdots x_n
    yzt^\omega.$$
    If\/ \pv W is $\tau$-tame with respect to~\Cl C, then so is $\pv
    R\vee\pv W$.
  \end{enumerate}
\end{theorem}

Theorem~\ref{t:ALCO-v_vs_tameness} yields for instance that the join
$\pv J\vee\pv G$ is graph $\kappa$-tame, a result which was also
proved by Steinberg \cite{Steinberg:2001b}.

\section{Relatively free profinite semigroups}
\label{sec:ALCO-structure}

Several representation theorems and structural results about
relatively free profinite semigroups have been obtained for various
pseudovarieties, such as $\pv J$ (\Cl J-trivial)
\cite[Section~8.2]{Almeida:1994a}, $\pv R$%
\index{$\mathsf{R}$} %
(\Cl R-trivial) \cite{Almeida&Weil:1996b,Almeida&Zeitoun:2003b},
$\pv{DA}$%
\index{$\mathsf{DA}$} %
(regular elements are idempotent) \cite{Moura:2009a}, and $\pv {LSl}$
(local semilattices) \cite{Costa:2001a}. Much remains unknown,
particularly in the case of pseudovarieties containing $\pv {LSl}$.
However,
progress has been made in this case too. For instance,
in~\cite{Gool&Steinberg:2016,Almeida&ACosta&Costa&Zeitoun:2017}
faithful representations of finitely generated free profinite semigroups over $\pv A$  were obtained. There is a common trend
in these faithful representations
of free profinite semigroups over $\pv A$, $\pv R$, or $\pv {DA}$,
and also in the partial faithful representations obtained in~\cite{Huschenbett&Kufleitner:2014,Kufleitner&Wachter:2016}
for many other pseudovarieties:
it is the fact they consist in viewing pseudowords
as linearly ordered sets whose elements are labeled with letters, generalizing the fact that words are nothing else than such sets with a finite cardinal.

In the most general case, that of the pseudovariety $\pv S$ of
all finite semigroups, no meaningful faithful representation
is known (albeit we can always get partial
information on the elements of $\Om AS$ by looking
at their projection on $\Om AV$, for some semigroup
pseudovariety $\pv V$, when a suitable representation for $\Om AV$
is available).
This adds motivation for studying
the structure of free profinite semigroups over $\pv S$ and other ``large'' pseudovarieties.
In this section we review some results on this subject, mainly
about Green's relations, with an emphasis on maximal subgroups. A substantial part of the results originated in connections with symbolic dynamics, most introduced by the first author, sometimes in co-authorship.
We highlight some of the progress in this front.
Other approaches, for the most part developed by Rhodes and Steinberg,
based on expansions of finite
semigroups or on wreath product techniques, also led to results about
structural properties of free profinite semigroups over many
pseudovarieties containing $\pv {LSl}$, as is the case
in~\cite{Rhodes&Steinberg:2002,Steinberg:2010ISRJM,Steinberg:2010,ACosta&Steinberg:2011}.
We mention in Subsection~\ref{sec:ALCO-proj-prof} two results
where these other approaches played a key role, namely
Theorems~\ref{t:ALCO-maximal-subgroup-sofic-case}
and~\ref{t:ALCO-closed-subgroups-are-projective}.

\subsection{Connections with symbolic dynamics}
\label{sec:ALCO-symb-dyn}

For a good reference book on symbolic dynamics,
see~\cite{Lind&Marcus:1996}. Even though an introduction to symbolic
dynamics appears in Chapter~27, for the convenience of the
presentation we include our own brief introduction. Let $A$ be a
finite alphabet. Since $A$ is compact, the product space $\shi A$ is
compact. The \emph{shift} on $\shi A$ is the homeomorphism
$\sigma:\shi A\to\shi A$ sending $(x_i)_{i\in\Z}$ to
$(x_{i+1})_{i\in\Z}$. A \emph{symbolic dynamical system}, also called
\emph{shift space}%
\index{shift space} %
or \emph{subshift}%
\index{subshift}, %
is a nonempty\footnote{The empty set is frequently considered a
  subshift in the literature (\emph{e.g.}, in Chapter~27).}
closed subspace $\ci X$ of $\shi A$ such that $\sigma(\ci X)=\ci X$,
for some finite alphabet~$A$. A shift space $\ci X$ is \emph{minimal}%
\index{shift space!minimal --}%
\index{subshift!minimal --}, %
if $\ci X$ does not contain subshifts other than $\ci X$. A
\emph{block}%
\index{block} %
of $(x_i)_{i\in\Z}$ is a word $x_ix_{i+1}\cdots x_{i+n}$, with
$i\in\Z$ and $n\geq 0$. Denote by $\B(\ci X)$%
\index{$\mathscr{J}_{\mathrm{V}}(\mathcal{X})$} %
the set of all blocks of elements of~$\ci X$. One has $\ci X\subseteq \ci
Y$ if and only if $\B(\ci X)\subseteq \B(\ci Y)$.

Often, one may define a subshift by an effectively computable
amount of data. This happens for example if $\B(\ci X)$%
\index{$\mathcal{B}(\mathcal{X})$} %
is a rational language, in which case we say that $\ci X$ is
\emph{sofic}%
\index{shift space!sofic --}%
\index{subshift!sofic --}. %
Sofic subshifts are considered in
Chapter~27. Another class of
examples, extensively studied, comes from subshifts defined by
\emph{primitive substitutions}~\cite{Fogg:2002}. Here, by a
\emph{substitution}%
\index{substitution} %
over a finite alphabet $A$ we mean an endomorphism $\varphi$ of $A^+$.
A substitution $\varphi$ over $A$ is \emph{primitive}%
\index{primitive!substitution}%
\index{substitution!primitive --} %
if there is $n\ge 1$ such that all letters of $A$ are factors of
$\varphi^n(a)$, for every $a\in A$.
For such a primitive substitution, there is a unique minimal subshift $\ci
X_{\varphi}$ such that $\B(\ci X_{\varphi})$ is the set of all factors of
words of the form $\varphi^n(a)$, where $n\ge 1$ and $a\in A$.

A subset $L$ of a semigroup $S$ is \emph{irreducible} if $u,v\in L$ implies
$uwv\in L$ for some $w\in S$.
A subshift $\ci X$ of $\z A$ is
\emph{irreducible}%
\index{shift space!irreducible --}%
\index{subshift!irreducible --} %
if $\B(\ci X)$ is an irreducible language of $A^+$.
Minimal subshifts are irreducible.
A subshift $\ci X$ is
    \emph{periodic}%
\index{shift space!periodic --}%
\index{subshift!periodic --} %
if $\ci X$ is a finite set of the form $\{\sigma^n(x):n\in\Z\}$ for some
$x\in\ci X$. An irreducible subshift is either periodic or infinite.

For the remainder of this subsection, $\pv V$ is a pseudovariety
of semigroups containing all finite nilpotent semigroups.
Then $A^+$ is isomorphic with $\om AV$ and embeds in $\Om AV$. The
elements of $A^+$ are isolated in $\Om AV$. Hence $\overline{L}\cap
A^+=L$ holds for every language $L$ of~$A^+$. Therefore
$\overline{\B(\ci X)}$ captures all information about $\ci X$.
Clearly, $\B(\ci X)$ is closed under taking factors;
when $\pv V=\pv A\malcev\pv V$, 
the topological closure of $\B(\ci X)$ in $\Om AV$ is also closed under taking factors, a fact that follows from the multiplication being open in $\Om AV$
when $\pv V=\pv A\malcev\pv V$, cf.~\cite[Lemma 2.3 and Proposition 2.4]{Almeida&ACosta:2007a}.

Using a compactness argument, in case $\ci X$ is irreducible one
shows the existence of a unique $\leq_\J$-minimal
$\J$-class $\apex_{\pv V}(\ci X)$%
\index{$\mathscr{J}_{\mathrm{V}}(\mathcal{X})$} %
consisting of factors of $\overline{\B(\ci X)}$.
If $\pv V$ contains $\pv{LSl}$
then $\ov{\B(\ci X)}$ consists of
elements of $\Om AV$ whose finite factors belong to $\B(\ci
X)$~\cite{ACosta:2006}.
From this one gets the following proposition, which
a particular case of~\cite[Proposition 3.6]{ACosta&Steinberg:2011}.

\begin{proposition}\label{p:ALCO-order-j}
  Let $\pv V$ be a pseudovariety of semigroups
  containing $\pv{LSl}$.
  Let $\ci X $ and $\ci Y$ be irreducible subshifts.
  Then
  $\ci X\subseteq \ci Y$ if and only if\/
  $\apex_\pv V(\ci Y)\leq_{\J}\apex_\pv V(\ci X)$.
\end{proposition}

The following result is taken from~\cite{Almeida:2004a}.

\begin{theorem}\label{t:ALCO-maximal-regular-J-classes}
  If\/ $\pv V$ contains~$\pv{LSl}$, then the mapping $\ci X\mapsto
  \apex_{\pv V}(\ci X)$ is a bijection from the set of minimal
  subshifts of $\shi A$ onto the set of 
  $\le_\J$-maximal regular $\J$-classes of\/~\Om AV.
\end{theorem}

If $|A|\geq 2$, then there are $2^{\aleph_0}$ minimal subshifts of
$\shi A$ (cf.~\cite[Chapter~2]{Lothaire:2001}),
and a chain with $2^{\aleph_0}$ irreducible
subshifts of $\shi A$~\cite[Section 7.3]{Walters:1982}. Hence, from
Theorem~\ref{t:ALCO-maximal-regular-J-classes} and
Proposition~\ref{p:ALCO-order-j} we obtain the following
result (a weaker version appears in \cite{Costa:2001a}).

\begin{theorem}\label{t:ALCO-chains-antichains}
  Let $\pv V$ be a pseudovariety containing $\pv{LSl}$ and let $A$ be an
  alphabet with at least two letters. For the relation $<_{\J}$ in
  $\Om AV$, there are both chains and anti-chains with $2^{\aleph_0}$
  regular elements.
\end{theorem}

For a subshift $\ci X$ of $\shi A$, the real number $h(\ci
X)=\lim\frac{1}{n}\log_2|\B(\ci X)\cap A^n|$ is its \emph{entropy}%
\index{subshift!entropy}%
\index{shift space!entropy}%
\index{entropy}. %
This fundamental concept is also considered in Chapter~27.
For $\pv V$ containing \pv{LSl}, and for $w\in \Om AV$, let $q_w(n)$
be the number of finite factors of $w$ of length $n$. The real number
$h(w)=\lim \frac{1}{n}\log_2q_w(n)$ exists if $w\in\Om AV\setminus
A^+$; we define $h(w)=0$ for $w\in A^+$. The number $h(w)$ is also
called the~\emph{entropy of $w$}%
\index{pseudoword!entropy} %
and was used in~\cite{Almeida&Volkov:2002b}
to get structural information about $\Om AV$.

Note that $h(\ci X)\leq \log_2|A|$. One can also show that $h(\ci
X)=\log_2|A|$ if and only if $\ci X=\shi A$, for every subshift $\ci
X$ of $\shi A$. The following is a similar result.

\begin{proposition}[{\cite{Almeida&Volkov:2002b}}]
  \label{p:ALCO-maximal-entropy}
  Let $\pv V$ be a pseudovariety 
  containing $\pv{LSl}$.
  Suppose $|A|\geq 2$.
  Let $w\in\Om AV$.
  Then $h(w)\leq \log_2|A|$,
  and equality holds
  if and only if $w$ belongs to the minimum ideal
  of\/ $\Om AV$.
\end{proposition}

For each $k$ such that $0<k\leq \log_2|A|$, consider the set $E_k$ of
all $w\in\Om AV$ with $h(w)<k$. In particular, thanks to
Proposition~\ref{p:ALCO-maximal-entropy}, $E_{\log_2|A|}$ is the
complement of the minimum ideal of~$\Om AV$. The following summarizes
the most important results from~\cite{Almeida&Volkov:2002b}.

\begin{theorem}
  \label{t:ALCO-entropy-vs-product_and_iteration}
  Let $\pv V$ be a pseudovariety containing $\pv{LSl}$ and suppose 
  $0<k\leq \log_2|A|$.
  \begin{enumerate}
  \item For all $u,v\in \Om AV$, we have $h(uv)=\max\{h(u),h(v)\}$,
    and so $E_k$ is a subsemigroup of\/~\Om AV.
    In particular, the minimum ideal is prime.
  \item The set $E_k$ is stable under the application of every $n$-ary
    implicit operation $w$ such that $h(w)<k\cdot \log_2{|A|}\cdot
    \log_n{|A|}$.
  \item The set $E_k$ is also stable under the iterated application of
    a continuous endomorphism $\varphi$ of\/~$\Om AV$ such that
    $\varphi(A)\subseteq E_k$, in the following sense: if $\psi$
    belongs to the closed subsemigroup of $\End {\Om AV}$ generated by
    $\varphi$, then $\psi(E_k)\subseteq E_k$.
  \end{enumerate}
\end{theorem}

If $\ci Y$ is a proper subshift of an irreducible sofic subshift $\ci
X$, then $h(\ci Y)<h(\ci X)$, see~\cite[Corollary
4.4.9]{Lind&Marcus:1996}. By a reduction to this result, the following
theorem generalizing some of the above mentioned properties of the
minimum ideal of $\Om AV$ is proved in~\cite{ACosta&Steinberg:2011}.

\begin{theorem}
  \label{p:ALCO-entropy-sofic-case}
  Let $\pv V$ be a pseudovariety of semigroups containing $\pv{LSl}$
  and let \ci X be a sofic subshift of $A^{\mathbb Z}$. Suppose that
  $\pv V=\pv A\malcev\pv V$ or $\B(\ci X)$
  is \pv V-recognizable.
   Then $h(w)<h(\ci X)$ whenever $w\in\overline{\B(\ci
    X)}\setminus \apex_\pv V(\ci X)$. Moreover, $\Om AV\setminus
   \apex_\pv V (\ci X)$ is a subsemigroup of\/~$\Om AV$.
\end{theorem}

\subsection{Closed subgroups of relatively free profinite semigroups}
\label{sec:ALCO-proj-prof}

Note that maximal subgroups of profinite semigroups are closed. If a
closed subsemigroup of a profinite semigroup is a group then, for the
induced topology, it is a profinite group. This subsection presents
results on the structure of closed subgroups of relatively
free profinite semigroups, with an emphasis on maximal subgroups,
using symbolic dynamics.

We shall see examples of maximal subgroups that are (relatively) free
profinite groups. When $A$ is a finite set and $\pv H$ is a nontrivial
pseudovariety of groups, it is customary to refer to
the cardinal of $A$ as being the \emph{rank}%
\index{rank!of free profinite group}%
\index{free profinite group!rank} %
of $\Om AH$.

A \emph{retract}%
\index{retract}%
\index{semigroup!retract of profinite --} %
of $\Om AS$ is the image of a continuous idempotent endomorphism of
$\Om AS$.
The free profinite subgroups of $\Om AS$ of rank $|A|$ that are retracts of~$\Om AS$ are
characterized in~\cite[Theorem~4.4]{Almeida&Volkov:2002b}. Combining
that characterization with results from~\cite{Almeida:2004a}, leads to
the following theorem.

\begin{theorem}
  \label{t:ALCO-group-generic-retracts}
  For every finite alphabet $A$, there are maximal subgroups $H$
  of\/~$\Om AS$ such that $H$ is a retract of\/ $\Om AS$ and
  a free profinite group of rank $|A|$.
\end{theorem}

The maximal subgroups in a $\J$-class of a profinite semigroup are
isomorphic profinite groups
(cf.~\cite[Theorem~A.3.9]{Rhodes&Steinberg:2009qt}). When $\ci X$ is an
irreducible subshift, we may consider the (isomorphism class of the)
maximal profinite subgroup of $\apex_{\pv V}(\ci X)$, denoted $G_{\pv
  V}(\ci X)$. It is invariant under isomorphisms of subshifts, as long
as $\pv V=\pv V\ast\pv D$ and $\pv V$ contains all finite
semilattices~\cite{ACosta:2006}, where \pv D%
\index{$\mathsf{D}$} %
denotes the pseudovariety of all finite semigroups whose idempotents
are right zeros.

Let $\varphi:A^+\to A^+$ be a primitive substitution. The substitution
$\varphi$ is called \emph{periodic}%
\index{substitution!periodic} %
if the associated minimal subshift~$\Cl X_\varphi$ is periodic. If
there are $b,c\in A$ such that $\varphi(a)$ starts with $b$ and ends
with $c$ for every $a\in A$, then the substitution is said to
be~\emph{proper}%
\index{substitution!proper}. %
Denote respectively by $\varphi_{\pv S}$ and by $\varphi_{\pv G}$ the
unique extension of $\varphi$ to a continuous endomorphism of $\Om AS$
and to a (continuous) endomorphism of $\Om AG$. The following is a
result from~\cite{Almeida&ACosta:2013}.

\begin{theorem}\label{t:ALCO-proper-primitive}
  If $\varphi$ is a proper non-periodic primitive substitution over
  $A$, then the retract $\varphi_{\pv S}^\omega(\Om AS)$ is a maximal
  subgroup of $\apex_{\pv S}(\ci X_\varphi)$, which is presented as a
  profinite group by the set of generators $A$ subject to the
  relations of the form $\varphi_{\pv G}^\omega(a)=a$ ($a\in A$).
\end{theorem}

For the general case where $\psi$ is a primitive (not necessarily
proper) non-periodic substitution, one finds in
\cite{Durand&Host&Skau:1999} an algorithm to build a proper primitive
substitution $\varphi$ such that $\Cl X_{\varphi}$ is isomorphic to
$\Cl X_{\psi}$, and so the general case can be reduced to the proper
case via the invariance of the maximal subgroup under isomorphism of
subshifts. An alternative finite presentation for $G_{\pv S}(\ci
X_{\psi})$ as a profinite group is given
in~\cite{Almeida&ACosta:2013}. These results yield that it is
decidable whether a given finite group is a (continuous) homomorphic
image of $G_{\pv S}(\ci X_{\psi})$.

Note that, if in Theorem~\ref{t:ALCO-proper-primitive} the extension
of~$\varphi$ to the free group over $A$ is invertible, then we
immediately get that $G_{\pv S}(\ci X_{\varphi})$ is a free profinite
group of rank $|A|$, which is a particular case
of~\cite[Corollary~5.7]{Almeida:2004a}. On the other hand, it was
proved in~\cite{Almeida&ACosta:2013} that if $\tau$ is
the~\emph{Prouhet-Thue-Morse substitution}%
\index{substitution!proper}, %
that is, the substitution given by $\tau(a)=ab$ and $\tau(b)=ba$, then
$G_{\pv S}(\ci X_{\tau})$ is not a relatively free profinite group.

In~\cite{Almeida&ACosta:2016b}, further knowledge on $G_\pv S(\Cl X)$ was obtained, when $\Cl X$
is minimal, without requiring that $\Cl X$ is defined by a substitution.
Namely, it was shown that $G_\pv S(\Cl X)$ is an inverse limit of profinite
completions of fundamental groups of a special family of finite
graphs that is naturally associated to $\Cl X$. 

For the sofic case, concerning groups of the form $G_{\pv V}(\ci X)$,
we have first to introduce a definition which is similar with the
definition, given in Section~\ref{sec:ALCO-profinite-top-algebras}, of
a free profinite group over a pseudovariety $\pv H$ of groups
(cf.~\cite[Chapter 3]{Ribes&Zalesskii:2000}). A subset $X$ of a
profinite group $G$ is said to~\emph{converge to the identity} if
every neighborhood of the identity element of $G$ contains all but
finitely many elements of $X$. A profinite group $F$ is
a %
\index{basis!converging to the identity} %
\emph{free pro-$\pv H$ group with a
  basis $X$ converging to the identity} if $X$ is a subset of $F$ for
which every mapping $\varphi:X\to G$, with $G$ a pro-$\pv H$ group
such that $\varphi(X)$ converges to the identity, has a unique
extension to a continuous group homomorphism $\hat\varphi:F\to G$. All
bases of $F$ converging to the identity have the same cardinality,
which is called the \emph{rank}%
\index{rank!of free profinite group}%
\index{free profinite group!rank} %
of $F$, and if $F$ and $F'$ have bases converging to the identity with
the same cardinal, then $F$ and $F'$ are isomorphic as profinite
groups. Note that, if $|X|$ is finite, then $F$ is isomorphic
with~$\Om XH$, and so this definition of rank extends the one given
for finitely generated relatively free profinite groups. A free
pro-\pv H group in the former sense is also free pro-\pv H with some
basis converging to one, but the converse is not true; indeed, as
follows from Theorem~\ref{t:ALCO-maximal-subgroup-sofic-case}, for a
nontrivial pseudovariety of groups \pv H, the free pro-\pv H group of
countable rank is metrizable, but \Om AH is not metrizable when $A$ is
infinite (see note following
Proposition~\ref{p:ALCO-free-pro-Q-through-completion}).

For a pseudovariety $\pv H$ of groups, denote
by $\HH$ the pseudovariety of all finite semigroups whose subgroups belong to
$\pv H$. Note that $\pv S=\GG$.
We are now able to cite the
result from~\cite{ACosta&Steinberg:2011}
about maximal subgroups of the from $G_{\HH}(\ci X)$.
Note that the minimal sofic subshifts
are periodic subshifts. 

\begin{theorem}
   \label{t:ALCO-maximal-subgroup-sofic-case}
   Let \pv H be a nontrivial pseudovariety of groups and $\ci X$ an
   irreducible sofic subshift. If $\ci X$ is periodic, then
   $G_{\HH}(\ci X)$ is a free pro-$\pv H$ group of rank $1$. If $\ci
   X$ is non-periodic and $\B(\ci X)$ is $\HH$-recognizable then
   $G_{\HH}(\ci X)$ is a free pro-$\pv H$ group of rank $\aleph_0$,
   provided $\pv H$ is extension-closed and contains nontrivial
   $p$-groups for infinitely many primes $p$.
 \end{theorem}

 Note that Theorem~\ref{t:ALCO-maximal-subgroup-sofic-case} applies to
$\ci X=\shi A$, in which case $\apex_{\HH} (\ci X)$ is the minimum ideal
of $\Om A{\HH}$. This case was previously shown in~\cite{Steinberg:2010ISRJM}.
For further results on the structure of the minimum ideal of $\Om AV$,
where $\pv V$ may be among pseudovarieties other than those in
Theorem~\ref{t:ALCO-maximal-subgroup-sofic-case},
see~\cite{Rhodes&Steinberg:2002,Steinberg:2010ISRJM}.
In contrast with Theorems~\ref{t:ALCO-proper-primitive}
and~\ref{t:ALCO-maximal-subgroup-sofic-case},
the $\ci H$-class of a non-regular element of\/~$\Om AS$ is
a singleton~\cite[Corollary 13.2]{Rhodes&Steinberg:2002}.

While not all closed subgroups of $\Om AS$ are free profinite groups,
they do have a property resembling freeness. A profinite group $G$ is
\emph{projective}%
\index{projective} %
if, for all continuous onto homomorphisms of profinite groups
$\varphi:G\to K$ and $\alpha:H\to K$, there is a continuous
homomorphism $\hat{\varphi}:G\to H$ such that
$\alpha\circ\hat{\varphi}=\varphi$. It is easy to see that all
(finitely generated) projective profinite groups embed into (finitely
generated) free profinite groups~\cite{Ribes&Zalesskii:2000}. The
following converse is much more difficult to prove.

\begin{theorem}[{\cite{Rhodes&Steinberg:2008}}]\label{t:ALCO-closed-subgroups-are-projective}
  Let $\pv V$ be a pseudovariety of semigroups such that
  $(\pv V\cap \pv {Ab})\ast \pv V=\pv V$. Then every closed subgroup
  of a free pro-$\pv V$ semigroup is a projective profinite group.
\end{theorem}

The definition of projective profinite group can be
considered for other algebras.
The~\emph{projective profinite semigroups}%
\index{semigroup!projective profinite --} %
embed into free profinite semigroups
but, in contrast with Theorem~\ref{t:ALCO-closed-subgroups-are-projective},
there are finite subsemigroups of $\Om AS$ that are not projective.
For further details about finite subsemigroups
of $\Om AS$ (and $\Om AV$ for other $\pv V$), and their
interplay with projective profinite semigroups,
see~\cite[Remark 4.1.34]{Rhodes&Steinberg:2009qt}.

\bibliographystyle{abbrv}
\addcontentsline{toc}{section}{References}
\begin{footnotesize}
  \bibliography{abbrevs,profinite-topologies}

\newcommand{\noopsort}[1]{} \newcommand{\singleletter}[1]{#1}
  \newcommand{\etal}{et al.}
\begin{thebibliography}{100}

\bibitem{Albert&Baldinger&Rhodes:1992}
D.~Albert, R.~Baldinger, and J.~Rhodes.
\newblock The identity problem for finite semigroups (the undecidability of).
\newblock {\em J. Symbolic Logic}, 57:179--192, 1992.

\bibitem{Alibabaei:2017b}
K.~Alibabaei.
\newblock Every decidable pseudovariety of abelian groups is completely tame.
\newblock Technical report, CMUP, Univ. Porto, 2017.
\newblock arXiv:1707.09131.

\bibitem{Alibabaei:2017c}
K.~Alibabaei.
\newblock The pseudovariety of all nilpotent groups is tame.
\newblock Technical report, CMUP, Univ. Porto, 2017.
\newblock arXiv:1712.09547.

\bibitem{Almeida:1989b}
J.~Almeida.
\newblock Residually finite congruences and quasi-regular subsets in uniform
  algebras.
\newblock {\em Portugal. Math.}, 46:313--328, 1989.

\bibitem{Almeida:1990c}
J.~Almeida.
\newblock On pseudovarieties, varieties of languages, filters of congruences,
  pseudoidentities and related topics.
\newblock {\em Algebra Universalis}, 27:333--350, 1990.

\bibitem{Almeida:1994a}
J.~Almeida.
\newblock {\em Finite semigroups and universal algebra}.
\newblock World Scientific, Singapore, 1995.
\newblock {E}nglish translation.

\bibitem{Almeida:1999b}
J.~Almeida.
\newblock Hyperdecidable pseudovarieties and the calculation of semidirect
  products.
\newblock {\em Internat. J. Algebra Comput.}, 9:241--261, 1999.

\bibitem{Almeida:1996d}
J.~Almeida.
\newblock Some algorithmic problems for pseudovarieties.
\newblock {\em Publ. Math. (Debrecen)}, 54 Suppl.:531--552, 1999.

\bibitem{Almeida:1999c}
J.~Almeida.
\newblock Dynamics of implicit operations and tameness of pseudovarieties of
  groups.
\newblock {\em Trans. Amer. Math. Soc.}, 354:387--411, 2002.

\bibitem{Almeida:2002a}
J.~Almeida.
\newblock Finite semigroups: an introduction to a unified theory of
  pseudovarieties.
\newblock In G.~M.~S. Gomes, J.-E. Pin, and P.~V. Silva, editors, {\em
  Semigroups, algorithms, automata and languages}, pages 3--64, Singapore,
  2002. World Scientific.

\bibitem{Almeida:2004a}
J.~Almeida.
\newblock Profinite groups associated with weakly primitive substitutions.
\newblock {\em Fundamentalnaya i Prikladnaya Matematika (Fundamental and
  Applied Mathematics)}, 11(3):13--48, 2005.
\newblock In Russian. English version in J. Math. Sciences \textbf{144}, No. 2
  (2007) 3881--3903.

\bibitem{Almeida:2003cshort}
J.~Almeida.
\newblock Profinite semigroups and applications.
\newblock In V.~B. Kudryavtsev and I.~G. Rosenberg, editors, {\em Structural
  theory of automata, semigroups and universal algebra}, pages 1--45, New York,
  2005. Springer.

\bibitem{Almeida&Borlido:2017}
J.~Almeida and C.~Borlido.
\newblock Complete $\kappa$-reducibility of pseudovarieties of the form drh.
\newblock {\em Internat. J. Algebra Comput.}, 27:189--235, 2017.

\bibitem{Almeida&ACosta:2007a}
J.~Almeida and A.~Costa.
\newblock Infinite-vertex free profinite semigroupoids and symbolic dynamics.
\newblock {\em J. Pure Appl. Algebra}, 213:605--631, 2009.

\bibitem{Almeida&ACosta:2013}
J.~Almeida and A.~Costa.
\newblock Presentations of {S}ch\"utzenberger groups of minimal subshifts.
\newblock {\em Israel J. Math.}, 196:1--31, 2013.

\bibitem{Almeida&Costa:2015a}
J.~Almeida and A.~Costa.
\newblock A note on pseudovarieties of completely regular semigroups.
\newblock {\em Bull. Austral. Math. Soc.}, 92:233--237, 2015.

\bibitem{Almeida&ACosta:2016b}
J.~Almeida and A.~Costa.
\newblock A geometric interpretation of the {S}chützenberger group of a minimal
  subshift.
\newblock {\em Ark. Mat.}, 54:243--275, 2016.

\bibitem{Almeida&ACosta&Costa&Zeitoun:2017}
J.~Almeida, A.~Costa, J.~C. Costa, and M.~Zeitoun.
\newblock The linear nature of pseudowords.
\newblock Technical report, 2017.
\newblock arXiv:1702.08083 [cs.FL].

\bibitem{Almeida&Costa&Teixeira:2010}
J.~Almeida, J.~C. Costa, and M.~L. Teixeira.
\newblock Semidirect product with an order-computable pseudovariety and
  tameness.
\newblock {\em Semigroup Forum}, 81:26--50, 2010.

\bibitem{Almeida&Costa&Zeitoun:2004}
J.~Almeida, J.~C. Costa, and M.~Zeitoun.
\newblock Tameness of pseudovariety joins involving {R}.
\newblock {\em Monatsh. Math.}, 146:89--111, 2005.

\bibitem{Almeida&Costa&Zeitoun:2005b}
J.~Almeida, J.~C. Costa, and M.~Zeitoun.
\newblock Complete reducibility of systems of equations with respect to {R}.
\newblock {\em Portugal. Math.}, 64:445--508, 2007.

\bibitem{Almeida&Costa&Zeitoun:2014}
J.~Almeida, J.~C. Costa, and M.~Zeitoun.
\newblock Closures of regular languages for profinite topologies.
\newblock {\em Semigroup Forum}, 80:20--40, 2014.

\bibitem{Almeida&Delgado:1997}
J.~Almeida and M.~Delgado.
\newblock Sur certains syst\`emes d'\'equations avec contraintes dans un groupe
  libre.
\newblock {\em Portugal. Math.}, 56:409--417, 1999.

\bibitem{Almeida&Delgado:1999}
J.~Almeida and M.~Delgado.
\newblock Sur certains syst\`emes d'\'equations avec contraintes dans un groupe
  libre---addenda.
\newblock {\em Portugal. Math.}, 58:379--387, 2001.

\bibitem{Almeida&Delgado:2001}
J.~Almeida and M.~Delgado.
\newblock Tameness of the pseudovariety of abelian groups.
\newblock {\em Internat. J. Algebra Comput.}, 15:327--338, 2005.

\bibitem{Almeida&Klima:2017a}
J.~Almeida and O.~Kl{\'\i}ma.
\newblock Towards a pseudoequational proof theory.
\newblock Technical report, Univ. Masaryk and Porto, 2017.
\newblock arXiv:1708.09681.

\bibitem{Almeida&Shahzamanian&Steinberg:2017}
J.~Almeida, M.~H. Shahzamanian, and B.~Steinberg.
\newblock The pro-nilpotent group topology on a free group.
\newblock {\em J. Algorithms}, 480:332--345, 2017.

\bibitem{Almeida&Silva:1997a}
J.~Almeida and P.~V. Silva.
\newblock On the hyperdecidability of semidirect products of pseudovarieties.
\newblock {\em Comm. Algebra}, 26:4065--4077, 1998.

\bibitem{Almeida&Silva:1997b}
J.~Almeida and P.~V. Silva.
\newblock {SC}-hyperdecidability of {{\bf R}}.
\newblock {\em Theoret. Comput. Sci.}, 255:569--591, 2001.

\bibitem{Almeida&Steinberg:2000a}
J.~Almeida and B.~Steinberg.
\newblock On the decidability of iterated semidirect products and applications
  to complexity.
\newblock {\em Proc. Lond. Math. Soc.}, 80:50--74, 2000.

\bibitem{Almeida&Steinberg:2000b}
J.~Almeida and B.~Steinberg.
\newblock Syntactic and global semigroup theory, a synthesis approach.
\newblock In J.~C. Birget, S.~W. Margolis, J.~Meakin, and M.~V. Sapir, editors,
  {\em Algorithmic problems in groups and semigroups}, pages 1--23. Birkhäuser,
  2000.

\bibitem{Almeida&Steinberg:2008}
J.~Almeida and B.~Steinberg.
\newblock Rational codes and free profinite monoids.
\newblock {\em J. London Math. Soc. (2)}, 79:465--477, 2009.

\bibitem{Almeida&Trotter:1999a}
J.~Almeida and P.~G. Trotter.
\newblock Hyperdecidability of pseudovarieties of orthogroups.
\newblock {\em Glasgow Math. J.}, 43:67--83, 2001.

\bibitem{Almeida&Trotter:2001}
J.~Almeida and P.~G. Trotter.
\newblock The pseudoidentity problem and reducibility for completely regular
  semigroups.
\newblock {\em Bull. Austral. Math. Soc.}, 63:407--433, 2001.

\bibitem{Almeida&Volkov:2002b}
J.~Almeida and M.~V. Volkov.
\newblock Subword complexity of profinite words and subgroups of free profinite
  semigroups.
\newblock {\em Internat. J. Algebra Comput.}, 16:221--258, 2006.

\bibitem{Almeida&Weil:1994}
J.~Almeida and P.~Weil.
\newblock Relatively free profinite monoids: an introduction and examples.
\newblock In J.~B. Fountain, editor, {\em Semigroups, formal languages and
  groups}, volume 466, pages 73--117, Dordrecht, 1995. Kluwer Academic Publ.

\bibitem{Almeida&Weil:1996b}
J.~Almeida and P.~Weil.
\newblock Free profinite {$\cal R$}-trivial monoids.
\newblock {\em Internat. J. Algebra Comput.}, 7:625--671, 1997.

\bibitem{Almeida&Weil:1996}
J.~Almeida and P.~Weil.
\newblock Profinite categories and semidirect products.
\newblock {\em J. Pure Appl. Algebra}, 123:1--50, 1998.

\bibitem{Almeida&Zeitoun:1998}
J.~Almeida and M.~Zeitoun.
\newblock The pseudovariety {{\bf J}} is hyperdecidable.
\newblock {\em Theor. Inform. Appl.}, 31:457--482, 1997.

\bibitem{Almeida&Zeitoun:2003b}
J.~Almeida and M.~Zeitoun.
\newblock An automata-theoretic approach of the word problem for $\omega$-terms
  over {R}.
\newblock {\em Theoret. Comput. Sci.}, 370:131--169, 2007.

\bibitem{Ash:1991}
C.~J. Ash.
\newblock Inevitable graphs: a proof of the type {II} conjecture and some
  related decision procedures.
\newblock {\em Internat. J. Algebra Comput.}, 1:127--146, 1991.

\bibitem{Auinger&Steinberg:2001b}
K.~Auinger and B.~Steinberg.
\newblock On the extension problem for partial permutations.
\newblock {\em Proc. Amer. Math. Soc.}, 131:2693--2703, 2003.

\bibitem{Birkhoff:1935}
G.~Birkhoff.
\newblock On the structure of abstract algebras.
\newblock {\em Proc. Cambridge Phil. Soc.}, 31:433--454, 1935.

\bibitem{Birkhoff:1937}
G.~Birkhoff.
\newblock {M}oore-{S}mith convergence in general topology,.
\newblock {\em Ann. of Math. (2)}, 38(1):39--56, 1937.

\bibitem{Burris&Sankappanavar:1981}
S.~Burris and H.~P. Sankappanavar.
\newblock {\em A Course in Universal Algebra}.
\newblock Number~78 in Grad. Texts in Math. Springer, Berlin, 1981.

\bibitem{Clark&Davey&Freese&Jackson:2004}
D.~Clark, B.~A. Davey, R.~S. Freese, and M.~Jackson.
\newblock Standard topological algebras: syntactic and principal congruences
  and profiniteness.
\newblock {\em Algebra Universalis}, 52:343--376, 2004.

\bibitem{ACosta:2006}
A.~Costa.
\newblock Conjugacy invariants of subshifts: an approach from profinite
  semigroup theory.
\newblock {\em Internat. J. Algebra Comput.}, 16:629--655, 2006.

\bibitem{ACosta&Steinberg:2011}
A.~Costa and B.~Steinberg.
\newblock Profinite groups associated to sofic shifts are free.
\newblock {\em Proc. Lond. Math. Soc.}, 102(2):341--369, 2011.

\bibitem{Costa:2001a}
J.~C. Costa.
\newblock Free profinite locally idempotent and locally commutative semigroups.
\newblock {\em J. Pure Appl. Algebra}, 163:19--47, 2001.

\bibitem{Costa&Nogueira:2009}
J.~C. Costa and C.~Nogueira.
\newblock {Complete reducibility of the pseudovariety {LSl}}.
\newblock {\em Internat. J. Algebra Comput.}, 19(2):247--282, 2009.

\bibitem{Costa&Teixeira:2004}
J.~C. Costa and M.~L. Teixeira.
\newblock Tameness of the pseudovariety {LSl}.
\newblock {\em Internat. J. Algebra Comput.}, 14:627--654, 2004.

\bibitem{Coulbois&Khelif:1999}
T.~Coulbois and A.~Khélif.
\newblock Equations in free groups are not finitely approximable.
\newblock {\em Proc. Amer. Math. Soc.}, 127:963--965, 1999.

\bibitem{Delgado:2001}
M.~Delgado.
\newblock On the hyperdecidability of pseudovarieties of groups.
\newblock {\em Internat. J. Algebra Comput.}, 11:753--771, 2001.

\bibitem{Delgado&Masuda&Steinberg:2007}
M.~Delgado, A.~Masuda, and B.~Steinberg.
\newblock Solving systems of equations modulo pseudovarieties of abelian groups
  and hyperdecidability.
\newblock In {\em Semigroups and formal languages}, pages 57--65. World Sci.
  Publ., Hackensack, NJ, 2007.

\bibitem{Durand&Host&Skau:1999}
F.~Durand, B.~Host, and C.~Skau.
\newblock Substitutional dynamical systems, {B}ratteli diagrams and dimension
  groups.
\newblock {\em Ergodic Theory Dynam. Systems}, 19(4):953--993, 1999.

\bibitem{Eilenberg:1976}
S.~Eilenberg.
\newblock {\em Automata, languages and machines}, volume~B.
\newblock Academic Press, New York, 1976.

\bibitem{Fogg:2002}
N.~P. Fogg.
\newblock {\em Substitutions in dynamics, arithmetics and combinatorics},
  volume 1794 of {\em Lecture Notes in Mathematics}.
\newblock Springer-Verlag, Berlin, 2002.

\bibitem{Gehrke&Grigorieff&Pin:2008}
M.~Gehrke, S.~Grigorieff, and J.-E. Pin.
\newblock Duality and equational theory of regular languages.
\newblock In {\em ICALP 2008, Part II}, number 5126 in Lecture Notes in
  Computer Science, pages 246--257. Springer Verlag, 2008.

\bibitem{Gehrke&Grigorieff&Pin:2010}
M.~Gehrke, S.~Grigorieff, and J.-E. Pin.
\newblock A topological approach to recognition.
\newblock In {\em ICALP 2008, Part II}, number 6199 in Lecture Notes in
  Computer Science, pages 151--162. Springer Verlag, 2010.

\bibitem{Gool&Steinberg:2016}
S.~J.~v. Gool and B.~Steinberg.
\newblock Pro-aperiodic monoids via saturated models.
\newblock Technical report, 2016.
\newblock arXiv:1609.07736.

\bibitem{Hall:1950}
M.~Hall.
\newblock A topology for free groups and related groups.
\newblock {\em Ann. of Math.}, 52:127--139, 1950.

\bibitem{Henckell:1988}
K.~Henckell.
\newblock Pointlike sets: the finest aperiodic cover of a finite semigroup.
\newblock {\em J. Pure Appl. Algebra}, 55:85--126, 1988.

\bibitem{Henckell&Rhodes:1991}
K.~Henckell and J.~Rhodes.
\newblock The {t}heorem of {K}nast, the {PG}={BG} and {T}ype {II}
  {C}onjectures.
\newblock In J.~Rhodes, editor, {\em Monoids and semigroups with applications},
  pages 453--463, Singapore, 1991. World Scientific.

\bibitem{Henckell&Rhodes&Steinberg:2010}
K.~Henckell, J.~Rhodes, and B.~Steinberg.
\newblock Aperiodic pointlikes and beyond.
\newblock {\em Internat. J. Algebra Comput.}, 20:287--305, 2010.

\bibitem{Herwig&Lascar:1997}
B.~Herwig and D.~Lascar.
\newblock Extending partial automorphisms and the profinite topology on free
  groups.
\newblock {\em Trans. Amer. Math. Soc.}, 352:1985--2021, 2000.

\bibitem{Hunter:1983}
R.~P. Hunter.
\newblock Some remarks on subgroups defined by the {B}ohr compactification.
\newblock {\em Semigroup Forum}, 26:125--137, 1983.

\bibitem{Hunter:1988}
R.~P. Hunter.
\newblock Certain finitely generated compact zero-dimensional semigroups.
\newblock {\em J. Austral. Math. Soc. Ser. A}, 44:265--270, 1988.

\bibitem{Huschenbett&Kufleitner:2014}
M.~Huschenbett and M.~Kufleitner.
\newblock Ehrenfeucht-{F}ra\"iss\'e games on omega-terms.
\newblock In {\em LIPIcs. Leibniz Int. Proc. Inform.}, volume~25, pages
  374--385. STACS, Schloss Dagstuhl. Leibniz-Zent. Inform., Wadern, 2014.

\bibitem{Karnofsky&Rhodes:1982}
J.~Karnofsky and J.~Rhodes.
\newblock Decidability of complexity one-half for finite semigroups.
\newblock {\em Semigroup Forum}, 24:55--66, 1982.

\bibitem{Kharlampovich&Sapir:1994}
O.~G. Kharlampovich and M.~Sapir.
\newblock Algorithmic problems in varieties.
\newblock {\em Internat. J. Algebra Comput.}, 5:379--602, 1995.

\bibitem{Krohn&Rhodes:1968}
K.~Krohn and J.~Rhodes.
\newblock Complexity of finite semigroups.
\newblock {\em Ann. of Math. (2)}, 88:128--160, 1968.

\bibitem{Kufleitner&Wachter:2016}
M.~Kufleitner and J.~P. W\"achter.
\newblock The word problem for omega-terms over the {T}rotter-{W}eil hierarchy
  (extended abstract).
\newblock In {\em Computer science---theory and applications}, volume 9691 of
  {\em Lecture Notes in Comput. Sci.}, pages 237--250. Springer, [Cham], 2016.

\bibitem{Lind&Marcus:1996}
D.~Lind and B.~Marcus.
\newblock {\em An introduction to symbolic dynamics and coding}.
\newblock Cambridge University Press, Cambridge, 1995.

\bibitem{Lothaire:2001}
M.~Lothaire.
\newblock {\em Algebraic combinatorics on words}.
\newblock Cambridge University Press, Cambridge, UK, 2002.

\bibitem{Makanin:1977}
G.~S. Makanin.
\newblock The problem of solvability of equations in a free semigroup.
\newblock {\em Mat. Sbornik (N.S.)}, 103 (2):147--236, 1977.
\newblock In Russian. English translation in: {{\sl Math. USSR-Sb.}} 32 (1977)
  128-198.

\bibitem{Margolis&Sapir&Weil:1999}
S.~Margolis, M.~Sapir, and P.~Weil.
\newblock Closed subgroups in pro-{V} topologies and the extension problem for
  inverse automata.
\newblock {\em Internat. J. Algebra Comput.}, 11:405--445, 2001.

\bibitem{McCammond:1999a}
J.~McCammond.
\newblock Normal forms for free aperiodic semigroups.
\newblock {\em Internat. J. Algebra Comput.}, 11:581--625, 2001.

\bibitem{Molchanov:1994}
V.~A. Molchanov.
\newblock Nonstandard characterization of pseudovarieties.
\newblock {\em Algebra Universalis}, 33:533--547, 1995.

\bibitem{Moura:2009a}
A.~Moura.
\newblock Representations of the free profinite object over {DA}.
\newblock {\em Internat. J. Algebra Comput.}, 21:675--701, 2011.

\bibitem{Nikolov&Segal:2003}
N.~Nikolov and D.~Segal.
\newblock Finite index subgroups in profinite groups.
\newblock {\em C. R. Acad. Sci. Paris S\'er. A.}, 337:303--308, 2003.

\bibitem{Numakura:1957}
K.~Numakura.
\newblock Theorems on compact totally disconnetced semigroups and lattices.
\newblock {\em Proc. Amer. Math. Soc.}, 8:623--626, 1957.

\bibitem{Perkins:1968}
P.~Perkins.
\newblock Bases for equational theories of semigroups.
\newblock {\em J. Algebra}, 11:298--314, 1968.

\bibitem{Pin:1991}
J.-E. Pin.
\newblock Topologies for the free monoid.
\newblock {\em J. Algebra}, 137:297--337, 1991.

\bibitem{Pin:1995d}
J.-E. Pin.
\newblock Eilenberg's theorem for positive varieties of languages.
\newblock {\em Russian Math. (Iz. VUZ)}, 39:74--83, 1995.

\bibitem{Pin&Reutenauer:1991}
J.-E. Pin and C.~Reutenauer.
\newblock A conjecture on the {H}all topology for the free group.
\newblock {\em Bull. Lond. Math. Soc.}, 23:356--362, 1991.

\bibitem{Pin&Silva:2011}
J.-E. Pin and P.~V. Silva.
\newblock On profinite uniform structures defined by varieties of finite
  monoids.
\newblock {\em Internat. J. Algebra Comput.}, 21(1-2):295--314, 2011.

\bibitem{Pin&Silva:2014}
J.-E. Pin and P.~V. Silva.
\newblock A noncommutative extension of {M}ahler's theorem on interpolation
  series.
\newblock {\em European J. Combinatorics}, 36:564--578, 2014.

\bibitem{Pin&Weil:1996a}
J.-E. Pin and P.~Weil.
\newblock Profinite semigroups, {M}al'cev products and identities.
\newblock {\em J. Algebra}, 182:604--626, 1996.

\bibitem{Pin&Weil:1996b}
J.-E. Pin and P.~Weil.
\newblock A {R}eiterman theorem for pseudovarieties of finite first-order
  structures.
\newblock {\em Algebra Universalis}, 35:577--595, 1996.

\bibitem{Place&Zeitoun:2014b}
T.~Place and M.~Zeitoun.
\newblock Going higher in the first-order quantifier alternation hierarchy on
  words.
\newblock In J.~Esparza, P.~Fraigniaud, T.~Husfeldt, and E.~Koutsoupias,
  editors, {\em Automata, languages, and programming. {P}art {II}
  ({ICALP'14})}, volume 8573 of {\em Lecture Notes in Comput. Sci.}, pages
  342--353, 2014.

\bibitem{Place&Zeitoun:2014a}
T.~Place and M.~Zeitoun.
\newblock Separating regular languages with first-order logic.
\newblock In {\em {CSL-LICS'14}}, pages 75:1--75:10. ACM, 2014.
\newblock DOI 10.1145/2603088.2603098.

\bibitem{Polak:2004}
L.~Pol\'ak.
\newblock {A classification of rational languages by semilattice-ordered
  monoids.}
\newblock {\em Arch. Math., Brno}, 40(4):395--406, 2004.

\bibitem{Reiterman:1982}
J.~Reiterman.
\newblock The {B}irkhoff theorem for finite algebras.
\newblock {\em Algebra Universalis}, 14:1--10, 1982.

\bibitem{Rhodes&Steinberg:2002}
J.~Rhodes and B.~Steinberg.
\newblock Profinite semigroups, varieties, expansions and the structure of
  relatively free profinite semigroups.
\newblock {\em Internat. J. Algebra Comput.}, 11:627--672, 2002.

\bibitem{Rhodes&Steinberg:2008}
J.~Rhodes and B.~Steinberg.
\newblock Closed subgroups of free profinite monoids are projective profinite
  groups.
\newblock {\em Bull. Lond. Math. Soc.}, 40(3):375--383, 2008.

\bibitem{Rhodes&Steinberg:2009qt}
J.~Rhodes and B.~Steinberg.
\newblock {\em The $q$-theory of finite semigroups}.
\newblock Springer Monographs in Mathematics. Springer, 2009.

\bibitem{Ribes&Zalesskii:1993a}
L.~Ribes and P.~A. Zalesski{\u\i}.
\newblock On the profinite topology on a free group.
\newblock {\em Bull. Lond. Math. Soc.}, 25:37--43, 1993.

\bibitem{Ribes&Zalesskii:1993b}
L.~Ribes and P.~A. Zalesski{\u\i}.
\newblock The pro-$p$ topology of a free group and algorithmic problems in
  semigroups.
\newblock {\em Internat. J. Algebra Comput.}, 4:359--374, 1994.

\bibitem{Ribes&Zalesskii:2000}
L.~Ribes and P.~A. Zalesski{\u\i}.
\newblock {\em Profinite groups}.
\newblock Number~40 in Ergeb. Math. Grenzgebiete 3. Springer, Berlin, 2000.

\bibitem{Sapir:1988a}
M.~V. Sapir.
\newblock On the finite basis property for pseudovarieties of finite
  semigroups.
\newblock {\em C. R. Acad. Sci. Paris S\'er. I Math.}, 306:795--797, 1988.

\bibitem{Steinberg:1998a}
B.~Steinberg.
\newblock Inevitable graphs and profinite topologies: some solutions to
  algorithmic problems in monoid and automata theory, stemming from group
  theory.
\newblock {\em Internat. J. Algebra Comput.}, 11:25--71, 2001.

\bibitem{Steinberg:2001b}
B.~Steinberg.
\newblock On algorithmic problems for joins of pseudovarieties.
\newblock {\em Semigroup Forum}, 62:1--40, 2001.

\bibitem{Steinberg:2010}
B.~Steinberg.
\newblock A combinatorial property of ideals in free profinite monoids.
\newblock {\em J. Pure Appl. Algebra}, 214(9):1693--1695, 2010.

\bibitem{Steinberg:2010ISRJM}
B.~Steinberg.
\newblock Maximal subgroups of the minimal ideal of a free profinite monoid are
  free.
\newblock {\em Israel J. Math.}, 176:139--155, 2010.

\bibitem{Steinberg:2010pm}
B.~Steinberg.
\newblock On the endomorphism monoid of a profinite semigroup.
\newblock {\em Portugal. Math.}, 68:177--183, 2011.

\bibitem{Steinby:1992}
M.~Steinby.
\newblock A theory of tree language varieties.
\newblock In {\em Tree automata and languages ({L}e {T}ouquet, 1990)},
  volume~10 of {\em Stud. Comput. Sci. Artificial Intelligence}, pages 57--81.
  North-Holland, Amsterdam, 1992.

\bibitem{Volkov:1995}
M.~V. Volkov.
\newblock On a class of semigroup pseudovarieties without finite pseudoidentity
  basis.
\newblock {\em Internat. J. Algebra Comput.}, 5:127--135, 1995.

\bibitem{Walters:1982}
P.~Walters.
\newblock {\em An Introduction to ergodic theory}.
\newblock Number~79 in Grad. Texts in Math. Springer-Verlag, New York, 1982.
\newblock First softcover print 2000.

\bibitem{Willard:1970}
S.~Willard.
\newblock {\em General topology}.
\newblock Addison-Wesley, Reading, Mass., 1970.

\end{thebibliography}
\end{footnotesize}

\end{document}